\documentclass[onefignum,onetabnum]{siamonline220329}

\usepackage[margin=1in]{geometry}
\usepackage{xcolor,graphicx,chemarr,amsmath,amssymb,mathrsfs,subfigure,tikz}
\usetikzlibrary{arrows}

\newcommand{\pic}[2]{\includegraphics[scale=#1]{#2}}
\newcommand{\picc}[2]{\centering\pic{#1}{#2}}
\newcommand{\minp}[2]{\begin{minipage}{#1\textwidth}#2\end{minipage}}

\usepackage{wrapfig}

\graphicspath{{fig/}}

\newcommand\bx{{\bar x}}

\newcommand{\rank}{\mathop{\mbox{rank}}}

\newcommand{\lra}{\mathop{\longrightarrow}\limits}

\newcommand{\diag}{\mathop{\mbox{diag}}}

\newcommand{\tV}{{\widetilde{V}}}
\newcommand{\hV}{{\widehat{V}}}
\newcommand{\cV}{{\mathcal{V}}}

\newcommand{\R}{{\rm \bf R}}
\renewcommand{\Im}{{\mathop{\rm Im\,}}}
\newcommand{\Rs}{{\mathscr R}}
\newcommand{\Ss}{{\mathscr S}}
\newcommand{\Ns}{{\mathscr N}}
\newcommand{\Cs}{\mathscr C}
\newcommand{\tl}{\tilde}
\newcommand\ve{\varepsilon}
\newcommand{\wt}{\widetilde}
\newcommand\graphicalLF{Graphical Lyapunov Function }
\newcommand\comment[1]{}
\definecolor{lightgray}{rgb}{0.895,.895,.895}
\newcommand\mybox{}

\renewcommand{\ker}{{\mathop{\rm Ker\,}}}
\newcommand{\Imgamma}{\Im \Gamma}
\newcommand{\imA}{\Im A}
\newcommand{\kerGamma}{\ker \Gamma}
\newcommand{\kerA}{\ker A}
\newcommand{\kerV}{\ker V}
\newcommand{\kertV}{\ker \tV}
\newcommand{\kerQ}{\ker Q}

\newtheorem{remark}{Remark}%


\usepackage{lmodern}
\usepackage{anyfontsize}

\begin{document}
\title{\Large\bfseries\sffamily On structural contraction   of  biological interaction networks}
\author{M. Ali Al-Radhawi, David Angeli, and Eduardo D. Sontag}
\date{\today}
\maketitle

\begin{abstract}
	Biological networks are customarily described as \emph{structurally robust}. This means that they often function extremely well under large forms of perturbations affecting both the concentrations and the kinetic parameters. In order to explain this property, various mathematical notions have been proposed in the literature. In this paper, we propose the notion of \textit{structural contractivity}, building on the previous work of the authors. That previous work characterized the long-term dynamics of classes of Biological Interaction Networks (BINs), based on “rate-dependent Lyapunov functions”. Here, we show that stronger notions of convergence can be established by proving {structural contractivity} with respect to non-standard polyhedral $\ell_\infty$-norms. In particular, we show that such networks are nonexpansive. With additional verifiable conditions, we show that they are strictly contractive over arbitrary positive compact sets. In addition, we show that such networks entrain to periodic inputs. We illustrate our theory with examples drawn from the modeling of intracellular signaling pathways.
\end{abstract}

\begin{keywords}
Structural robustness, Stability, Contraction, Biological Interaction Networks, Systems Biology, Entrainment 
\end{keywords}

\begin{MSCcodes}
93D05 , 80A30, 47H09, 37B25, 05C90, 92B99, 34A34, 34D30, 37C20
\end{MSCcodes}
 \section{Introduction}
 
 The last quarter of a century witnessed an accelerating development of dynamical models and methods for the analysis and synthesis of biological systems, largely enabled by advances in data collection and processing \cite{alon06,ingalls_book,klipp16}. Applications range from unveiling complex biology \cite{MA22_pnas}, understanding cancer dynamics \cite{kreeger10}, precision medicine \cite{chelliah21}, and drug development \cite{sorger11}, to epidemiology \cite{avram24}.
The scope of such models includes interactions in biology at all scales, from the molecular, to the cellular, tissue, organism, and population level, as all these interactions share dynamical features that can be often analyzed using a common mathematical language.

One of the most important aspects of biological systems is that they tend to self-regulate exceedingly well against large uncertainties. An example of this is homeostasis, understood as the maintenance of a desired steady state in the face of internal fluctuations or exogenous perturbations. A complementary aspect is the ability to accurately respond to specific time-varying external signals while ignoring such internal fluctuations and irrelevant external perturbations. The roles of the system-level interactions and feedback loops that underlie such regulation have long been identified in biology \cite{umbarger56,savageau72}, and they have been better understood through the application of novel data collection and experimental measurements in the last 25 or so years. In conjunction, novel mathematical methods helped reveal that the regulatory architectures employed in biology are indeed inherently robust against variations in parameters and \emph{in-vivo} concentrations, in contrast to relying upon the fine-tuning of parameters and concentrations \cite{barkai97,alon99,stelling04}. In fact, robustness has been proposed as a key \emph{defining} property of biological networks \cite{morohashi02,kitano04}. A common framework to explain such robustness postulates that biology employs robust \emph{motifs} \cite{prill05}, which means that the uniform qualitative behavior is endowed by the network structure regardless of the parameters  \cite{bailey01,gupta22,araujo23}.
 
However, despite initial success, the rational analysis and synthesis of dynamical robustness in biology is considerably harder to achieve than in engineered systems. This is due to many factors.  First, biology pervasively employs nonlinear \textit{binding} events to relay information. Examples include the binding of enzymes to substrates, ligands to receptors, transcription factors to promoters, small molecules to proteins, etc. Such events are intrinsically nonlinear, and linearization is often unhelpful. Second, biological networks suffer from severe forms of uncertainty. For instance, the kinetics governing the speed of binding events and biochemical transformations are seldom known.  Such uncertainty is not limited to the difficulty of measuring the relevant parameters and identifying the functional forms, but also includes the fact that biological networks operate in highly variable environments. This lack of certainty precludes the availability of complete mathematical models.

For generic uncertain nonlinear models, it can easily be the case that small fluctuations in concentrations, or small variations in kinetic parameters, might lead to large effects causing trajectories to deviate from a desired region in the state space, and/or lose stability altogether, for example, becoming oscillatory or chaotic.  Although such  bifurcations or ``phase transitions'' can be desirable in some cases (e.g., genetic oscillators \cite{elowitz00} or toggle switches \cite{collins00}), they can cause the system to malfunction and/or make key species reach \emph{unsafe} levels. In fact, disease may be often defined mathematically as the loss of stability of a healthy phenotype \cite{maclean15,langlois17}. Thus, it is desirable to develop
robust techniques that allow for deriving quantitative and qualitative conclusions about a system of interest without \textit{a priori} knowledge of the kinetics, and help to identify those architectures for which a change in parameters cannot lead to a significant change in the character of global dynamic behavior. Such efforts  have been advocated as ``complex  biology without parameters'' \cite{bailey01}.

\paragraph{A useful formalism}

A very useful class of models in biology can be formulated in the language of \textit{Biological Interaction Networks} (BINs). This class of systems is also known as \textit{Chemical Reaction Networks} (CRNs), but we prefer to de-emphasize the ``chemical'' aspect because these models apply at all levels of scale. To give a concrete example of what we mean by this, consider the ``chemical reaction'' represented formally as follows:
\[
S+K \;\leftrightharpoons\; C \; \to \; P + K
\]
which can represent processes as different as biomolecular, epidemiology, or ecological interactions, Fig.~\ref{fig:examples_enzymatic}.
We formally define BINs in Section~\ref{sec:BIN}.

\begin{figure}[ht]
\minp{0.32}{\picc{0.75}{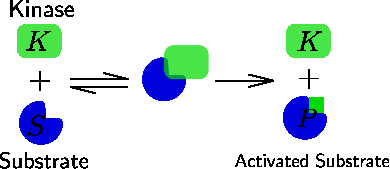}}%
\minp{0.3}{\picc{0.8}{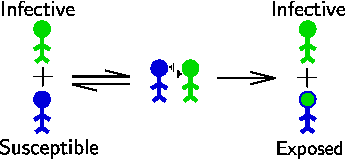}}%
\minp{0.38}{\picc{0.18}{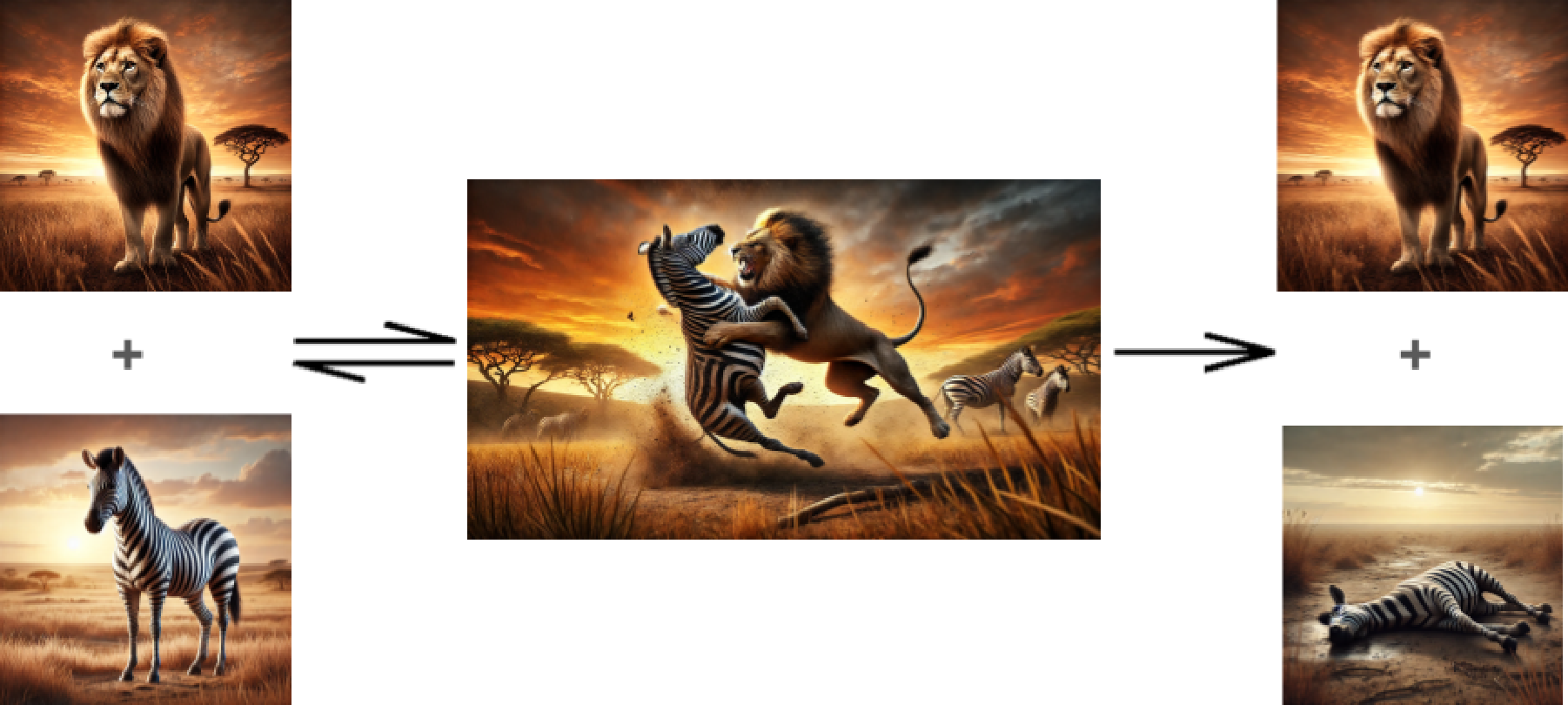}}

\null\ 

\minp{0.32}{\centerline{(a)}}%
\minp{0.3}{\centerline{(b)}}%
\minp{0.38}{\centerline{(c)}}%

\caption{The same BIN formalism 
$S+K \leftrightharpoons C \to P + K$
can represent processes as varied as:
(a) an interaction between a catalyst (in the figure, an enzyme kinase) and its substrate, in which the two species can temporarily interact in a reversible fashion, sometimes resulting in the catalyst being released and the substrate being modified (``product'');
(b) an interaction between an infective individual and a susceptible individual, in which the two individuals meet and either go their way without any new infection, or result in the susceptible individual becoming exposed to the disease;
or
(c) a predator/prey interaction in which the prey might not be harmed or the pr{e}y may end up injured or dead.}
\label{fig:examples_enzymatic}
\end{figure}

\paragraph{Notions of robustness} 

By \textit{structural dynamical robustness}, we mean, informally, that a BIN obeys some rigid constraints on its dynamics owing to its graphical structure, regardless of  kinetic parameters or the functional form of the kinetics.

Several notions have been proposed for the study of structural dynamic robustness, and despite the perceived difficulty, the effort has been strikingly successful for some classes of nonlinear BINs. An early example was Horn-Jackson-Feinberg (HJF) theory \cite{horn72,feinberg87}, where structural stability was shown for complex-balanced networks with mass-action kinetics. The method uses the sum of all the chemical pseudo-energies stored in species as a Lyapunov function, and it can establish \textit{global stability} in many cases \cite{sontag01,anderson11}.  Robust local stability has also been studied \cite{clarke80,blanchini19}. Another notion is \emph{structural persistence} which precludes species from asymptotically going extinct during the course of reaction, and has been shown for networks lacking a specific motif known as a \textit{critical siphon} \cite{angeli07p}. Another notion is \emph{structural monotonicity} which ensures that trajectories preserve a partial order, which in turn guarantees generic convergence under mild conditions \cite{angeli10,banaji13}.  
Yet another notion is that of \emph{robust non-oscillation}, where a test was provided for a given network not to admit oscillations \cite{MA_TAC22}. 

The direction which is most relevant to our paper relies upon the notion of structural stability of BINs as certified via piecewise linear (PWL) Lyapunov functions in reaction \cite{PWLRj,MA_LEARN} or concentration coordinates \cite{MA_cdc14,blanchini14,MA_LEARN}.   Such functions can either be constructed graphically \cite{PWLRj,blanchini14,MA_LEARN} or computationally \cite{MA_MCSS23}, and a dedicated computational package is available \cite{MA_LEARN}. In addition to establishing convergence, such Lyapunov functions allow the construction of ``safe sets'' that entrap the trajectories of the dynamical system for all future times, thus guaranteeing that undesirable regions of state space cannot be reached.
Compared to HJF theory, the PWL functions proposed in \cite{PWLRj,MA_LEARN} have a different functional form, distinct stability properties, 
and they are applicable to a large class of kinetics that are not mass-action. 
The main subject of this paper is to explore implications of PWL Lyapunov functions %
to structural \emph{contraction}.
 
\paragraph{Contraction} 
  
Informally, a dynamical system is \emph{strictly contractive} with respect to a given metric if it satisfies a strong form of incremental stability: it is required that the  distance between any two trajectories converges exponentially to zero, with no overshoot. A weaker version is that of a \textit{nonexpansive} system, in which this distance is only required to be non-increasing. 
Contraction can be compared with Lyapunov global stability, where the latter property merely requires that all trajectories converge to a steady-state. In globally asymptotically stable systems, the distance between any two trajectories can overshoot, and might not exponentially decrease to zero. Therefore, contraction is a much stronger notion than global stability. Contraction is also stricter than global \textit{exponential} stability to an equilibrium with no overshoot, as studied in some detail in \cite{duvall_sontag_2023_arxiv_ges_contractions}.
Early results on contractive systems were developed in the late 1950s using the notion of logarithmic norms \cite{dahlquist58,lozinskii58}. More recently, there has been renewed interest in this form of analysis \cite{lohmiller98,nijmeijer06,sontag14,forni14,aminzare_sontag_pde2015,davydov22,bullo22}.

Contraction offers many advantages. First, it has been argued \cite{bullo22} that the properties enjoyed by contractive systems are more akin to linear systems, and hence they offer a powerful general framework to analyze a much wider class of nonlinear systems with the same ease that one can analyze linear systems. Second, contractive systems \emph{entrain to periodic inputs} \cite{lohmiller98,sontag10,margaliot16}, which means that the state of a contractive system will asymptotically oscillate with the same frequency as a periodic forcing input. This property is important to biological systems, which often need to entrain to external inputs such as the circadian rhythm, or upstream events such as the cell cycle. Third, using contraction analysis methods, a system with unbounded trajectories can sometimes be shown to be nonexpansive, which greatly constrains its unstable behavior, while Lyapunov analysis does not offer further information when the trajectories are unbounded. 
 
In the context of BINs, contraction has been only studied with respect to standard norms or with respect to diagonal scaling of such norms \cite{DiBernardo10,aminzare_sontag_pde2015,vaghy23}. Contraction with respect to general polyhedral norms has been studied recently for the special case of systems that are monotone with respect to partial orders \cite{jafarpour22}.

\paragraph{This paper} 

We introduce the notion of \emph{structural contraction}, and show how this property can be deduced from the conditions for stability introduced in \cite{PWLRj,MA_LEARN} and implemented in the computational package \texttt{LEARN}. In particular, we show that the existence of the rate-dependent Lyapunov functions proposed in \cite{PWLRj,MA_LEARN} for a given {BIN} implies structural nonexpansivity with respect to weighted $\ell_\infty$ norms, and strict contractivity on positive compact sets under additional but easily verifiable conditions.  Furthermore, we show that such networks obey the notion of ``contraction after a small overshoot and short transient'' that was introduced in \cite{margaliot16}, and this fact allows us to show entrainment to periodic inputs. Therefore, the powerful tools of contraction theory can be applied to a rich class of BINs which can be characterized computationally via the MATLAB package \texttt{LEARN} (\texttt{github.com/malirdwi/LEARN}) or graphically \cite{MA_MCSS23}. %
  
 The paper is organized as follows. We first review the main results. Section 2 discusses preliminaries including extensions of previous results that are relevant to our current analysis. Section 3 shows that nonexpansivity follows from the existence of an appropriate rate-dependent Lyapunov function. Section 4 establishes contraction over positive compact sets. Section 5 shows  contraction after a small overshoot and short transient and  entrainment to periodic inputs. Section 6 discusses biochemical examples in detail, and conclusions are discussed in Section 7.
 
 \subsection{Overview of Results}

In this subsection, we offer an \textit{informal} guide that summarizes the main results of the paper.  They apply to BINs whose dynamics can be described by $\dot x=\Gamma R(x)$ where $\Gamma$ is the stoichiometry matrix and $R$ is the reaction rate. The only assumption needed is that the kinetics are monotone  (refer to \S 2 for the exact assumptions){.}

\paragraph{Review of PWL Lyapunov functions in rate} The results in \cite{PWLRj,MA_LEARN} provided functions that can be defined using the stoichiometry of the network without knowledge of the kinetics.  We first define a rate-dependent Graphical Lyapunov Function (GLF) $\tV(r)$, where $r$ has the dimension equal to the number of reactions, see Definition \ref{def.GLF}. It is characterized by the existence of a common Lyapunov function for a set of rank-one linear systems, see Theorem \ref{t.comLF}.
This means that $V(x)=\tV(R(x))$ is a Lyapunov function which is always non-increasing  along the dynamics, see Proposition \ref{lyap_dec}. In addition, if $\tV$ is invariant with respect to translation by $\Gamma$, then its existence is equivalent to the existence of a dual function $\hV$ satisfying $\hV(\Gamma r)=\tV(r)$,  see Theorem \ref{th.equiv}. Then, $\hV(x-\bx)$ is a Lyapunov function in the stoichiometric class compatible with a given steady state $\bx$, see Corollary \ref{cor.dualGLF_dec}. 

The computational algorithms are mainly developed for GLFs of the form $\tV(r)=\|C r \|_\infty$. For a given network, an $\ell_\infty$-GLF exists if a convex feasibility problem can be solved, see Corollary \ref{cor}. 
In addition, the dual GLF function is $\hV(z)=\|Bz\|_\infty$, see Corollary \ref{th.equiv_piecewise}.
 
\paragraph{Nonexpansivity and global convergence} 

By nonexpansivity with respect to a given norm $\|.\|_*$, we mean that the distance between two arbitrary trajectories $\|x_1(t)-x_2(t)\|_*$ does not increase over time. This can be characterized via non-positivity of the logarithmic norm of the Jacobian, see Theorem \ref{th.contraction}.

Our \textit{main result} states that if an $\ell_\infty$ GLF $\tV(r)=\|B\Gamma r\|_\infty$ exists for a network,
then the norm of the difference between two trajectories that lie in the same stoichiometry class is non-increasing along time,
see Theorem \ref{th.mainB}, thus establishing that the network is structurally nonexpansive with respect to the norm $\|.\|_B: z\mapsto \|Bz\|_{\infty}$. Furthermore, we provide an upper bound on the logarithmic norm of the Jacobian in terms of the partial derivatives of $R$, see Lemma \ref{mainlemmaB}.

In addition, we show that existence of a steady state automatically implies uniform boundedness, see Corollary \ref{cor.boundedness}. Ways to verify global stability are discussed in \S \ref{s.global}. 

\paragraph{Strict contraction over positive compact sets} 

In order to establish \textit{strict contraction} over positive compact sets, we utilize two graphical concepts. We first define trivial siphons (Definition \ref{d.trivialSiphon}), which help check when the trajectories of a network are uniformly separated from the boundary of the positive orthant, see Lemma \ref{eventuallyK}. Then, we introduce the concept of weak contraction of a network in terms of the connectivity of a specific matrix, see Definition \ref{def.weakContraction}.

Given the satisfaction of these two verifiable conditions, we show that a network is structurally strictly contractive over positive compact sets with respect to the norm $\|.\|_{PB}: z \mapsto \|PBz\|_\infty$, where $P$ is an appropriately-defined diagonal matrix that might depend on the choice of the compact set, see Theorem \ref{th.scalednorm}.

\paragraph{Entrainment} 

We then turn to the study of BINs where one or multiple parameters are periodic time-varying functions. We ask if all trajectories \emph{entrain} to the same period. To that end, we use our result establishing contraction over positive compact sets to show when a given BIN is also contractive with short transients and small overshoots, see Theorem \ref{th.glf_sost}. 
Once this is known, the results from \cite{margaliot16}, allow us to conclude that trajectories of a periodic time-varying system that stay in the same stoichiometric class entrain to periodic inputs, see Theorem \ref{th.entrainment}.

\section{Background and preliminaries}

\subsection{Notations}
 
For any set $S$, $|S|$ denotes its cardinality. The set of $n$-tuples of positive numbers is denoted by $\mathbb R_+^n$, while the set of $n$-tuples of non-negative numbers is denoted by $\mathbb R_{\ge0}^n$. The set of $n\times m$ matrices with real entries is denoted by $\mathbb R^{n\times m}$. For any subset $U\subset \mathbb R^{n\times m}$, its closure with respect to the standard Euclidean topology is denoted by $\bar U$.  If $A\in\mathbb R^{n\times m}$ is a matrix, then $\kerA$ is its kernel or null space, while $\imA$ is its image space. For any function $V:\mathcal X \to \mathbb R$, $\kerV:=\{x\in\mathcal X| V(x)=0\}$, and for a subset $\mathcal Y \subset \mathcal X$, $\kerV|_{\mathcal Y}:=\{x\in\mathcal Y| V(x)=0\}$. 
For $x=[x_1,\dots,x_n]^T \in \mathbb R^n$, $\|x\|_\infty:=\max_i |x_i|$. The notation $x\gg0$ means that $x  \in\mathbb R_+^n$, {while $x\ge0$ means that $x\in\mathbb R_{\ge0}^n$}. For a square matrix $A\in\mathbb R^{n\times n}$, $\sigma_i(A):= a_{ii}+\sum_{j\ne i} |a_{ij}|$, $i=1,\dots,n$, and $\mu_\infty(A):=\max_i \sigma_i(A)$. For a subspace $W \subset \mathbb R^n$, the quotient space $\mathbb R^n/W $ is the set of all sets of the form $\{x+y|y\in W\}$ where $x \in \mathbb R^n$. %
 
\subsection{Biological Interaction Networks}
\label{sec:BIN}
 
We review standard material \cite{feinberg87,erdi89,sontag01,MA_LEARN}.
 
\paragraph{Definition.} 

A Biological Interaction Network (BIN) 
$\mathscr N$ is specified as a pair $(\mathscr S,\mathscr R)$  where $\mathscr S:=\{X_1,\dots,X_n\}$ is the set of symbols called \textit{species}, and $\mathscr R:=\{\R_1,\dots,\R_\nu\}$ is a set of symbols called \textit{reactions}. Each reaction is associated to a formal rule of the following form:
 \[\R_j:~~ \sum_{i=1}^n {\alpha_{ij}}X_i \longrightarrow \sum_{i=1}^n \beta_{ij} X_i, \]
 where the nonnegative real numbers $\alpha_{ij},\beta_{ij}\ge 0$ are called the \textit{stoichiometric coefficients} of the respective reaction. 
 The $X_i$'s that appear with nonzero coefficients in the left-hand side are called the \textit{reactants} of the reaction $j$, and the $X_i$'s that appear with nonzero coefficients in the right-hand side are called the \textit{products} of the reaction $j$.
 In many applications the stoichiometric coefficients are integers, and one thinks of {the {$j$}th} reaction as saying that $\alpha_{{1j}}$ copies of $X_1$, \ldots,
 $\alpha_{{nj}}$ copies of $X_n$ combine to produce
 $\beta_{{1j}}$ copies of $X_1$, \ldots,
 $\beta_{{nj}}$ copies of $X_n$.
 The net ``loss'' (if negative) or ``gain'' (if positive) of the $i$th species in the $j$th reaction is  $\gamma_{ij}=\beta_{ij}-\alpha_{ij}$. These numbers can be collected into a matrix $\Gamma \in\mathbb R^{n\times \nu}$ defined entry-wise as $[\Gamma]_{ij}= \gamma_{ij}$.  The matrix $\Gamma$ is customarily called the \textit{stoichiometry matrix}.
 
 \paragraph{Kinetics.} 
 
 In order to quantify elements of the network, the species $X_1,\dots,X_n$ are assigned  non-negative numbers known as the \emph{concentrations} (representing the density of elements in a certain region of space) $x_1,\dots,x_n$, while reactions $\R_1,\dots,\R_\nu$ are assigned \textit{rates} $R_j:\mathbb R_{\ge 0}^n \to \mathbb R_{\ge0}, j=1,\dots,\nu$.  These rates describe ``how fast'' each reaction can occur, depending on the concentration of its reactants. The most common form of reaction rates is known as \textit{mass-action} kinetics, which are those of the form $R_j(x) = {k_j} \prod_{i=1,\alpha_{ij}>0}^n  {x_i}^{\alpha_{ij}}$, where $k_j>0$ is known as the kinetic constant. However, here we do not need to restrict to a specific form of kinetics. Instead we only assume that each such $R_j$ is monotone with respect to its reactants. More precisely, a reaction rate $R$ is \textit{admissible} if it satisfies the following general properties: \begin{itemize}
  \item[{\bf AK1:}] each reaction varies smoothly with respects to its reactants, i.e{.} $R(x)$ is $\mathscr C^1$;
  \item[{\bf AK2:}] a reaction requires all of its reactants to occur, i.e{.},  if $\alpha_{ij}>0$, then $x_i=0$ implies  $R_j(x)=0$; %
  \item[{\bf AK3:}] if a reactant increases, then the reaction rate increases, i.e{.} ${\partial R_j}/{\partial x_i}(x) \ge 0$ if $\alpha_{ij}>0$ and ${\partial R_j}/{\partial x_i}(x)\equiv 0$ if $\alpha_{ij}=0$. Furthermore, the aforementioned inequality is strict whenever the reactants are strictly positive.
 \end{itemize}
 In our paper, we assume every network $\Ns$ considered is equipped with a set of monotone kinetics satisfying AK1-AK3.   
 
 The admissible kinetics can be interpreted alternatively via a sign-pattern constraint on the Jacobian $\partial R/\partial x \in \mathbb R^{\nu\times n}$. More formally, let $\mathcal P_\Ns$ denote the set of reactant-reaction pairs, i.e., $\mathcal P_\Ns:=\{(i,j)| X_i~\mbox{is a reactant of}~\R_j\}$, and let $s=|\mathcal P_\Ns|$.  We define the set of all matrices that have an admissible sign pattern as \begin{equation}\label{K_def}\mathcal K_\Ns := \{ K \in \mathbb R^{\nu \times n} | [K]_{ji}>0~\mbox{if}~(i,j)\in\mathcal P_\Ns,~\mbox{and}~[K]_{ji}=0~\mbox{otherwise}\}. \end{equation} Therefore, when $R$ is admissible, we have  $\partial R/\partial x(x) \in \mathcal K_\Ns$ for $x\in\mathbb R_{+}^{n}$ and $\partial R/\partial x(x) \in \overline{\mathcal K_\Ns}$ for $x\in\mathbb R_{\ge 0}^{n}$.
 
Another equivalent formulation is obtained by decomposing $R$ into a conic sum of indicator matrices. To that end, let $E_{ji} \in \{0,1\}^{\nu \times n}$ be an indicator matrix, i.e{.}, $E_{ji}$ is zero everywhere except for the $(j,i)$th entry where it is equal to 1. In addition, let the elements of $\mathcal P_\Ns$ be indexed as $(i_1,j_1),\dots,(i_s,j_s)$.  Therefore, we can write \begin{equation}\label{e.calK} \mathcal K_\Ns=\{K \in \mathbb R^{\nu\times n}|K=\sum_{\ell=1}^s \bar\rho_\ell E_{j_\ell i_\ell},~\mbox{for some}~\bar\rho_1,\dots,\bar\rho_s>0\}.\end{equation} Hence, the admissibility of  $R$ implies the existence of continuous functions $\rho_1(x),\dots,\rho_s(x)\ge 0$ such that
\begin{equation}\label{e.dR}
 	\frac{\partial R}{\partial x}(x)= \sum_{\ell=1}^s \rho_\ell(x) E_{j_\ell i_\ell},
\end{equation}
and $\rho_\ell(x)>0$ whenever $x\in\mathbb R_+^n$.
 
\paragraph{Dynamics} 
 
The time evolution of the concentration vector $x=[x_1,\dots,x_n]^T$ as a function of time is by definition described by the following vector Ordinary Differential Equation (ODE):
\begin{equation}\label{e.ode}
 	\dot x \; = \; \Gamma R(x), \;x(0)=x_0.
\end{equation}
It is an easy exercise to show that this ODE is a \emph{positive system}, which means that the coordinates of the trajectory $x(t)$ remain non-negative if the initial condition $x(0)$ is non-negative. 
 
A non-zero vector $w\in\mathbb R_{\ge 0}^n/\{0\}$ is called a \textit{conservation law} if $w^T\Gamma=0$. If there is a $w$ that is strictly positive then the corresponding network is called \emph{conservative}. Note that this implies that $w^T\dot x(t)=0$ for all $t$, so $w^Tx$ is a conserved quantity that stays invariant, i.e., $w^Tx(t)\equiv w^T x(0)$. If such a $w$ exists, all trajectories will be confined to translates of the proper subspace $\Imgamma$ as can be seen by integrating \eqref{e.ode} to get $x(t)=x_0 + \Gamma \int_0^t R(x(\tau))\,d\tau$. Therefore, $x(t) \in \Cs_{x_0}$ for all $t\ge0$ where
\[
\Cs_{x_0}:= (\{x_0\} + \Imgamma)\cap \mathbb R_{\ge0}^n \,.
\]
The forward-invariant subset $\Cs_{x_0}$ is a closed and convex set known as the \textit{stoichiometric class} corresponding to $x_0$. {A stoichiometric class is proper if it intersects the open positive orthant. }

Suppose that $V:{\mathbb R}_{\ge0}^n \rightarrow {\mathbb R}$ is a differentiable function. Given any solution of a differential equation \eqref{e.ode}, the derivative of $V(x(t))$ with respect to $t$ can be expressed as $\dot V(x(t))$, where $\dot V$ is the function $(\partial V/\partial x) \Gamma R(x)$. This is the classical way to develop Lyapunov stability theory.  In this work,  we deal with locally Lipschitz but not necessarily differentiable functions $V$. For such functions, the gradient exists almost everywhere (Rademacher’s Theorem). For a global notion, we use the upper Dini derivative \cite{yoshizawa}: \begin{equation}\label{dini} \dot V(x):=\limsup_{h\to 0^+} \frac{V(x+h\Gamma R(x))-V(x)}h,\end{equation}  
which is defined everywhere for locally Lipschitz functions.
 
\paragraph{Fluxes}	A vector $v$ is called a \textit{flux} if $\Gamma v=0$.  In order to simplify the treatment, we will  assume the following about the stoichiometry of the network:
\begin{itemize}
\item[\textbf{AS1}:] 
There exists a positive flux, i.e.,  $\exists v \in \kerGamma$ such that $v\gg0$.  
\end{itemize}
Assumption AS1 is necessary for the existence of positive steady states for the corresponding dynamical system \eqref{e.ode}.

\paragraph{Extent-of-reaction system} 
 
A different representation of the dynamics uses the concept of \textit{extent-of-reaction} $\xi$ \cite{clarke80}. Given a concentration trajectory $x(t)$ and an arbitrary non-negative vector $\xi_0$, we may define 
\[
\xi(t)\,:=\;\int_0^t R(x(\tau))~d\tau+ \xi_0 \,.
\]
From this definition, $\dot \xi(t)=R(x(t))$ and $\xi(0)=\xi_0$.
On the other hand,
since $x(t)$ is a solution of \eqref{e.ode}, we know that 
\begin{equation}
\label{rel_x_xi} 
x(t)\;=\; x(0) + \Gamma \int_0^t R(x(\tau))\,d\tau \;=\;
x(0) + \Gamma (\xi(t) - \xi(0)). 
\end{equation}  
Using the latter expression, and writing $x_0=x(0)$, we can write the dynamics of $\xi$ as:
\[
\dot \xi \;=\; R(x_0 + \Gamma (\xi-\xi_0) )\,, \;\xi(0)=\xi_0\,,
\]
and specializing to the particular case $\xi_0=0$ we have that
$x(t) = x_0 + \Gamma (\xi(t))$ and
\begin{equation}
\dot \xi = R(x_0 + \Gamma \xi ) \,,\;\xi(0)=0. \label{e.xi_ode0}
\end{equation}
Conversely, suppose that $x_0$ is arbitrary and that $\xi$ is any solution of Equation \eqref{e.xi_ode0}, so that
$\xi(t) = \int_0^t R(x_0 + \Gamma \xi(\tau) )\,d\tau$.
Then we may define
\[
x(t)\,:=\; 
x_0 +  \Gamma \xi(t)\,.
\]
With this definition,
\[
\dot x(t) \;=\; \Gamma \dot\xi(t)
\;=\;
\Gamma R(x_0 + \Gamma \xi({t}) )
\;=\;
\Gamma R(x({t})).
\]
In other words, $x(t)$ is a solution of
Equation \eqref{e.ode}
with initial condition $x(0) = x_0+ \Gamma \xi(0) = x_0$.

In summary, there is a one to one correspondence between solutions of \eqref{e.ode} and solutions of \eqref{e.xi_ode0}. 
 
However, it is inconvenient to parameterize the extent-of-reaction ODE in terms of the initial condition $x_0$ as in \eqref{e.xi_ode0}.
Next, we derive another form.  Instead of considering an initial condition $x_0$, let us consider the corresponding stoichiometric class $\Cs_{x_0}$. 
Pick any point $\bx \in \Cs_{x_0}$. (For example, $\bx$ might be a steady state.) For any such choice of $\bx$, there exists some $\xi_0$ such that $x_0 - \bx = \Gamma \xi_0$.  Then, we manipulate the expression to get $x(t)-x_0=x(t)-\bx+\bx-x_0= x(t)-\bx - \Gamma \xi_0$. Substituting in \eqref{rel_x_xi}, we get $x(t)=\bx+\Gamma\xi(t)$. Since $\dot\xi=R(x)$, we can write:
\begin{equation}
\label{e.xiode} 
\dot\xi = R(\bx + \Gamma \xi),\; \xi(0)=\xi_0.
\end{equation}
We will use the latter equation in the remainder of the paper.
The alternative ODE \eqref{e.xiode} has the advantage of transferring the 
dependence on $x(0)$ to the initial condition $\xi(0)$. Hence, any trajectory $x^*(t)$ that starts in $\Cs_\bx$ can be recovered as $\bx+\Gamma \xi^* (t)$, where $\xi^*(t)$ the solution of \eqref{e.xiode} for an appropriate choice of $\xi^*(0)$. Note that $\xi^*(0)$ is non-unique since any other initial condition $\xi^{**}(0)$ satisfying $\Gamma(\xi^*(0))=\Gamma(\xi^{**}(0))$ will give the same trajectory. Hence, the extent-of-reaction system is known as a \emph{translation-invariant} system \cite{angeli08}, and can be thought as evolving on $\mathbb R^\nu/\kerGamma$.  This will be discussed more formally in \S \ref{sec.quotient}.
 
\paragraph{Example} Let us now consider a network which represents a simplified post-translational modification cycle which is shown in Figure \ref{f.networks}(a) later on in the paper. (Note: in order to simplify the calculations, we assume that $\bf R_1$ and $\bf R_2$ are irreversible, but the method applies to the original network.) The reactions are as follows:
\begin{equation}\label{e.ptm} \begin{array}{rl}  S+ E \xrightarrow{\R_1} C_1 \xrightarrow{\R_2} P+E \\ P+D \xrightarrow{\R_3} C_2 \xrightarrow{\R_4} S+D \end{array}\end{equation}
We define the state vector as $x=[s,e,c_1,p,d,c_2]^T$. The corresponding ODEs can be written as:
\begin{equation}
\label{e.ptm.eqn}
\dot x=\frac{d}{dt}\begin{bmatrix}s\\e\\c_1\\p\\d\\ c_2 \end{bmatrix}=\left[\begin{array}{rrrr} -1 & 0 & 0 & 1 \\ -1 & 1 & 0 & 0 \\  1 & -1 & 0 & 0 \\ 0 & 1 & -1 & 0  \\ 0 & 0 & -1 & 1 \\ 0 & 0 & 1 & -1 \end{array}\right]\begin{bmatrix}R_1(s,e)\\R_2(c_1)\\R_3(p,d)\\R_4(c_2) \end{bmatrix} = \Gamma R(x);~x(0)=x_0.
\end{equation}
The network admits three conservation laws, which are $s+c_1+p+c_2=s_{tot}$, $e+c_1=e_{tot}$, $d+c_2=d_{tot}$. Hence, any given vector $\bx\in\mathbb R_{\ge0}^6$ determines $s_{tot},e_{tot},d_{tot}$ and thus fixes a stoichiometric class $\Cs_\bx$. Any such class is a three dimensional, polyhedral, and compact subset.
 	
Recall that $R$ satisfies conditions AK1-AK3. Hence, its Jacobian $\partial{R}/{\partial x}\in\mathscr K_\Ns$, where $\mathcal K_\Ns$ is set of all matrices with the following sign pattern over the positive orthant:
\[
\begin{bmatrix}  + & + & 0 & 0 & 0 &0 \\ 0 & 0 & + & 0 & 0 & 0 \\ 0  & 0 & 0 & + & + & 0 \\ 0 & 0 & 0 & 0 & 0 &+ \end{bmatrix}
\,.
\]
Fixing any $\bx \in \Cs_{x_0}$, the corresponding extent-of-reaction dynamics can be written as:
\[
\dot\xi = \left[
\begin{array}{l} R_1(\bx_1-\xi_1+\xi_4,\bx_2-\xi_1+\xi_2)\\ R_2(\bx_3+\xi_1-\xi_2) \\ R_3(\bx_4+\xi_2-\xi_3,\bx_5-\xi_3+\xi_4)\\R_4(\bx_6+\xi_3-\xi_4)  
\end{array}
\right], \, \xi(0)=\xi_0 \,.
\]
The initial condition $\xi_0$ satisfies $x_0-\bx=\Gamma \xi_0$ as explained earlier after Eq. \eqref{e.xiode}.%
 
\subsection{Graphical Lyapunov Functions (GLFs)}
 
\paragraph{Example (continued)}

Let us consider \eqref{e.ptm.eqn} with an admissible $R$. Using results from \cite{MA_cdc14,PWLRj,MA_LEARN}, it can be shown that the following function: $V(x)=\max\mathcal R-\min\mathcal R$ \cite{MA_LEARN}, where $\mathcal R=\{R_1(s,e),R_2(c_1),R_3(p,d),R_4(c_2)\}$ is positive-definite with respect to the set of steady states and non-increasing for any choice of $R$.  In addition, note that it can be written as $V=\tV \circ R$, where $\tV$ is given as:
 \begin{equation}
 \label{e.tV_ptm}%
 \tV(r)=\| C r \|_\infty \;=\;
 \left\| \left[ 
 \begin{array}{rrrr}
 -1 & 0 & 0 & 1\\ -1 & 1 & 0 & 0\\ 1 & 0 & -1 & 0\\ 0 & 1 & -1 & 0\\ 0 & 1 & 0 & -1\\ 0 & 0 & 1 & -1 
 \end{array}
 \right] 
 r \right\|_\infty
 \,.
 \end{equation}
 
 In the above example, the function $\tV$ depends only on the graphical nature of a given network. We will call such a function a robust Lyapunov function or a graphical Lyapunov function. In order to provide a formal definition of $\tV$, we examine the non-increasing property for $V=\tV \circ R$. This amounts to:
 \[(\partial \tV/{\partial r})({\partial R}/{\partial x}) \Gamma R(x) \;\le\; 0
 \]
 for all $x$.  Since we need $\tV$ to work for every admissible rate $R$, and not to depend on the specific $R$, we introduce the stronger condition
 $
 (\partial \tV/{\partial r})K \Gamma r \le 0
 $
 for all $r$ and for any admissible Jacobian matrix $K\in\mathcal K_\Ns$.  This motivates the following definition:
 
 \begin{definition}
 \label{def.GLF} 
 Let $\Ns=(\Ss,\Rs)$ be given. A locally Lipschitz function $\tV: \mathbb R^\nu \to \mathbb R_{\ge 0}$ is a \graphicalLF (GLF) for  $\Ns$ if
 $\tV$ is positive-definite with respect to the set of steady states, meaning that
\[
\tV (r) \ge0  \, \forall r\in\mathbb R_{\ge0}^{\nu} , \; \tV(r) = 0\mbox{ iff } \Gamma r=0;
 \]
and
\[
\forall r\in\mathbb R_{\ge0}^{\nu}, \forall K \in \mathcal K_\Ns, \;\frac{\partial{\tV}}{\partial r}(r) K\Gamma r \le 0 \,.
\]
\end{definition}
 
 \begin{remark}\label{remark.Dini} Since $V$ is not assumed to be continuously differentiable, the gradient above is defined in terms of the upper Dini's derivative as $\frac{\partial{\tV}}{\partial r}(r) = [D_{e_1}^+\tV(r),\dots,D_{e_\nu}^+\tV(r)] $, where $\{e_j\}_{j=1}^\nu$ are the unit vectors and  $D_{e_j}^+ \tV(r) := \limsup_{h\to0^+} \tfrac 1h (\tV(r+h e_j) - \tV(r)) .$
 \end{remark}
 
 In order to characterize GLFs,	we write the non-increasing condition using \eqref{e.dR} as follows:
 \begin{equation}
 \label{e.rode} 
 \frac{\partial \tV}{\partial r}K\Gamma r \;=\;  \frac{\partial \tV}{\partial r}  \sum_{\ell=1}^s \bar\rho_\ell E_{ j_\ell i_\ell} \Gamma r \;=\; \sum_{\ell=1}^s \bar\rho_\ell \overbrace{(e_{j_\ell}\gamma_{i_\ell}^T)}^{Q_\ell} r
 \;=\;  \sum_{\ell=1}^s \bar\rho_\ell Q_\ell r\le 0\,,
 \end{equation}
 where  $e_j$ is a member of the standard basis of $\mathbb R^\nu$, and $\gamma_i$ is the $i$th row of $\Gamma$. 
 Since the coefficients $\bar\rho_1,\dots,\bar\rho_s$ are arbitrary nonnegative numbers, this motivates the following definition:
 
 \begin{definition}
 Let $\tV: \mathbb R^\nu\to\mathbb R_{\ge0}$ be a locally-Lipschitz function, and let $Q_1,\dots,Q_s \in \mathbb R^{\nu\times\nu}$. Then, we say that $\tV$ is a common Lyapunov function for the set of linear systems
 \[
 {\dot r}=Q_1 r, {\ldots},\dot r=Q_s r
 \]
 if
 \[
 \kertV=\bigcap_{\ell=1}^s \kerQ_\ell
 \]
 and 
 \[
 (\partial \tV/\partial r)Q_\ell r\le0 \mbox{  for all }\ell=1,\dots,s \,.
 \]
  \end{definition}
 We had previously shown the following characterization \cite{MA_cdc14,MA_LEARN}:
 \begin{theorem} \label{t.comLF}A function $\tV$ is {a} GLF for a given network iff $\tV$ is a common Lyapunov function for the set of linear systems $\dot r = Q_1 r, \dots, \dot r = Q_s r$. 
 \end{theorem}

 \mybox{\paragraph{Example (continued)}  We are still considering the network \eqref{e.ptm}. Note that we can write the following:
 	\begin{align*}\frac{\partial R}{\partial x}\Gamma&= {%
    \begin{bmatrix}\rho_1 & 0 & 0 & 0 & \rho_2 & 0 \\ 0 & 0 & \rho_3 & 0 & 0 & 0 \\
 				0 & \rho_4 & 0 & 0 & 0 & \rho_5 \\ 0 & 0 &0 	& \rho_6 & 0 & 0 \end{bmatrix}}{%
                \left[\begin{array}{rrrr} -1 & 0 & 0 & 1 \\ 0 & 1 & -1 & 0 \\ 1 & -1 & 0 & 0 \\ 0 & 0 & 1 & -1 \\ -1 & 1 & 0 & 0 \\ 0 & 0 & -1& 1 \end{array}\right]} \\  &= {%
                \rho_1\begin{bmatrix} -1 & 0 &0 & 1 \\ 0 & 0 & 0 & 0 \\ 0 & 0 & 0 & 0 \\  0 & 0 & 0 & 0 \end{bmatrix}+\rho_2 \begin{bmatrix} -1 & 1 & 0 & 0 \\ 0 & 0 & 0 & 0 \\ 0 & 0 & 0 & 0 \\  0 & 0 & 0 & 0 \end{bmatrix} + \rho_3 \begin{bmatrix} 0 & 0 & 0 & 0 \\ 1 & -1 & 0 & 0 \\ 0 & 0 & 0 & 0 \\  0 & 0 & 0 & 0 \end{bmatrix} } \\ & { %
                \quad + \; \rho_4 \begin{bmatrix} 0 & 0 & 0 & 0 \\ 0 & 0 & 0 & 0 \\ 0 & 1 & -1 & 0 \\  0 & 0 & 0 & 0 \end{bmatrix} +\rho_5 \begin{bmatrix} 0 & 0 & 0 & 0 \\ 0 & 0 & 0 & 0 \\  0 & 0 & -1 & 1  \\  0 & 0 & 0 & 0 \end{bmatrix}  +\rho_6  \begin{bmatrix} 0 & 0 & 0 & 0 \\ 0 & 0 & 0 & 0 \\    0 & 0 & 0 & 0 \\ 0 & 0 & 1 & -1  \end{bmatrix} } \\
 		&{%
= \; \sum_{\ell=1}^6 \rho_\ell Q_\ell }  
\end{align*}
It can be verified that the function $\tV$ defined in \eqref{e.tV_ptm} is indeed a common Lyapunov function for the set of rank-one linear systems $\dot r=Q_\ell r, \ell=1,\dots,6$.  %

\paragraph{Concentration dynamics} 

The next result shows that the GLF will give us a concrete Lyapunov function for each given $R$.
\begin{proposition}[\cite{MA_cdc14,MA_LEARN}] 
\label{lyap_dec}
Consider a network $\Ns=(\Ss,\Rs)$, and consider \eqref{e.ode} with a fixed admissible $R$. Assume that $\tV$ is a GLF for $\Ns$. Then, $V=\tV\circ R$ is Lyapunov function for \eqref{e.ode}, i.e., it is positive definite with respect to the set of steady states, and $\dot V(x)\le0$ for all $x$.
\end{proposition}
 
\subsection{The interplay between concentration and extent-of-reaction dynamical systems}
 
\subsubsection{Concentration dynamics as a quotient of extent-of-reaction dynamics} \label{sec.quotient}

The extent-of-reaction system has some peculiarities. Considering \eqref{e.xiode}, note that $\xi(0)=\xi^*+v$ will give rise to the same exact trajectory for any $v\in\kerGamma$. Moreover, when the system \eqref{e.ode} is at a steady-state, the system \eqref{e.xiode} exhibits an unbounded trajectory which is linear in time.  Therefore, in this subsection, we study  the relationship between concentration dynamics \eqref{e.ode} and extent-of-reaction dynamics \eqref{e.xiode} more formally.  Before proceeding, we restrict the class of functions $\tV$ under consideration using the following definition:
\begin{definition} 
A function $\tV:\mathbb R^\nu\to \mathbb R_{\ge0}$ is called  \emph{$\Gamma$-translation-invariant} if it satisfies $\tV(\xi)=\tV(\xi^*)$ whenever $\xi-\xi^* \in \kerGamma$.
\end{definition}
 
Fix a vector $\bx\in\mathbb R_{+}^n$.	Consider a system $\dot\xi=g(\xi)=R(\bx+\Gamma {\xi})$, $\xi\in\mathbb R^\nu$, and let $\psi_t:\mathbb R^\nu\to\mathbb R^\nu$ be the associated flow.
Similarly, consider the system $\dot x=f(x)=\Gamma R(x)$, $\xi\in\Cs_\bx:=\bx+\Imgamma$, and let $\varphi_t:\Cs_\bx\to\Cs_\bx$ be the associated flow.  Define the following mapping $\Phi:\mathbb R^\nu \to\Cs_\bx$:
\begin{equation}\label{Phi} 
\Phi(\xi):=\bx+\Gamma \xi. 
\end{equation}
Using \eqref{rel_x_xi}, we have $\Phi\circ\psi_t=\varphi_t$ for all $t\ge 0$. We are interested now in finding an inverse $\Psi:\Cs_\bx \to \mathbb R^\nu$. To that end, let us think of the stoichiometry matrix $\Gamma$ as a surjective linear operator $\Gamma:\mathbb R^\nu\to\Imgamma$. Hence, there exists another linear operator $G:\Imgamma\to\mathbb R^\nu$ which satisfies $\Gamma G(z)=z$ for all $z\in\Imgamma$.   %
In order to make $G$ into a bijection, consider the quotient space $\mathbb R^\nu/\kerGamma$, hence $\Gamma: \mathbb R^\nu/\kerGamma \to \Imgamma$ is a linear bijection, and let $G$ be its linear inverse.  Therefore $\Gamma \circ G = \mathrm{id}_{\Imgamma}$ and $G \circ \Gamma = \mathrm{id}_{\mathbb R^\nu/\kerGamma}$, where $\mathrm{id}_W$ denotes the identity operator on the vector space $W$.  Hence, we can write $\Psi:\Cs_\bx\to\mathbb R^\nu/\kerGamma$ as 
\begin{equation} \label{Psi}
\Psi(x)=G(x-\bx){.}
\end{equation} 
 
We can summarize our construction by saying that we have a \textit{bisimulation} relation, in the sense of \cite{schaft04,haghverdi05}, between the concentration and the extent-of-reaction dynamical systems as follows:
\begin{theorem}\label{bisimulation}
Let $\Ns=(\Ss,\Rs)$ be a network, and fix $\bx\in\mathbb R_{+}^n$.  Denote $\Cs_\bx:=\bx+\Imgamma$. Let $\varphi_t,\psi_t$ be the flows corresponding to \eqref{e.ode},\eqref{e.xiode}, respectively. Let $\Phi$ be defined as in \eqref{Phi}, and let $\Psi$ be given as in \eqref{Psi}.
Let $\tV:\mathbb R^\nu\to \mathbb R_{\ge0}$ be a locally Lipschitz function and  {$\Gamma$-translation-invariant}. Then, the following diagram commutes:
\begin{center}
\includegraphics[width=0.3\linewidth]{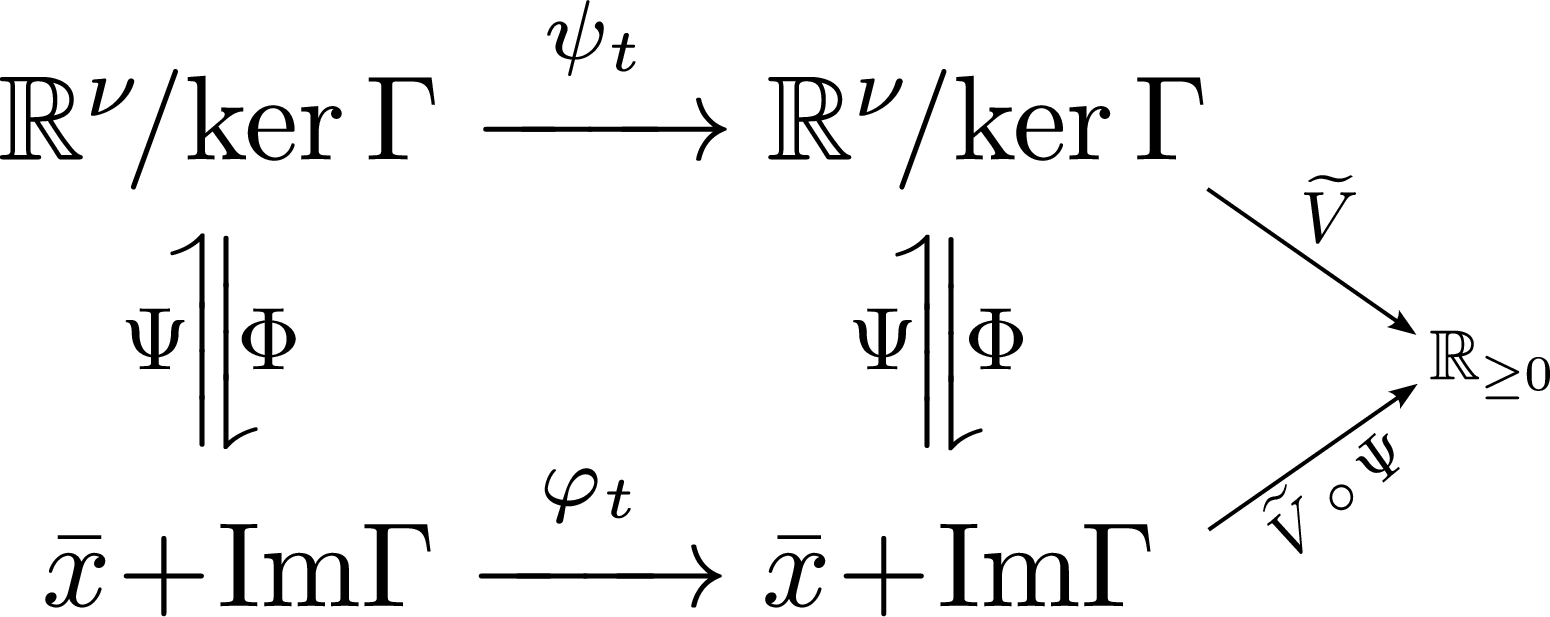}.
\end{center}
\end{theorem}
 
We introduce the concept of dual GLF:

\begin{definition}\label{def.hV_}
Given a network $\Ns$ and a $\Gamma${-translation}-invariant GLF $\tV$. Let $G:\Imgamma\to\mathbb R^\nu\!/\!\kerGamma$ be as defined above. Then, $\hV:=\tV\circ G$ is called a dual GLF.
\end{definition}
 
Using Theorem \ref{bisimulation}, we will shortly show in \S\ref{sec.dualGLF} that $(\tV\circ\Psi)(x)=\hV(x-\bx)$ is also a Lyapunov function on $\Cs_\bx$.
However, we will first establish that $\tV$ {is} a GLF for the extent-of-reaction dynamics, in a sense analogous to GLF's defined earlier.
 
\subsubsection{A GLF for the extent-of-reaction dynamics} 

We next define what it means for an extent-of-reaction system to admit a Lyapunov function.

\begin{definition}\label{def.xi_lyap}
Consider the extent-of-reaction system \eqref{e.xiode} $\dot\xi=R(\bx+\Gamma\xi)$. Then, a locally-Lipschitz function $\tV:\mathbb R^\nu\to\mathbb R_{\ge0}$ is said to be a Lyapunov function for \eqref{e.xiode} if:
\begin{enumerate}
  \item 
  $\tV(\xi)\ge0$ for all $\xi$,
  \item 
  $\tV(0)=0$,
  \item 
  $\tV$ is $\Gamma$-translation-invariant, and
  \item
  $\dot\tV(\xi)\le0$ for all $\xi$. 
\end{enumerate}
\end{definition}

As might be expected, we show that a GLF gives us a Lyapunov function for the extent-of-reaction system as stated below:

\begin{proposition}
\label{xi_GLF}
Consider a network $\Ns=(\Ss,\Rs)$,  and let $R$ be an admissible rate. Consider the extent-of-reaction system \eqref{e.xiode} $\dot\xi=R(\bx+\Gamma\xi)$. Assume that $\bx$ is a steady state for \eqref{e.ode}. If   $\tV$ is a $\Gamma$-translation{-}invariant GLF for $\Ns$, then $\tV$ {is} a Lyapunov function for the extent-of-reaction system \eqref{e.xiode}. 
\end{proposition}

\begin{proof}

We show the four statements given in Definition \ref{def.xi_lyap}. The first and second statements follow immediately from the corresponding properties of $\tV$ given in {Definition} \ref{def.GLF}. The third statement is assumed in the statement of the proposition. We prove the fourth statement next.

Since $\tV$ is a locally Lipschitz function, its gradient exists almost everywhere. At the points where it is defined, one write{s} $\dot \tV(\xi)$ as follows:
\[
\dot\tV(\xi)=\frac{\partial \tV}{\partial \xi} R(\bx+\Gamma \xi) \mathop{=}\limits^{(\star)} \frac{\partial \tV}{\partial \xi}  \left (R(\bx) + \frac{\partial R}{\partial x}(x_\xi^*)\Gamma \xi \right )\mathop{=}\limits^{(\star\star)}  \frac{\partial \tV}{\partial \xi}   \frac{\partial R}{\partial x}(x_\xi^*)\Gamma \xi  , 
\]
where $x_\xi^*:=\bx+\ve \Gamma \xi$ for some $\ve\in(0,1)$. The equality $(\star)$ above follows from applying the Mean-Value Theorem to the single-valued function $h(\ve)=R(\bx+\ve\Gamma \xi)$. 
The equality $(\star\star)$ can be shown as follows.
Since $\tV$ is $\Gamma$-translation-invariant, pick a $\hV$ such that $\hV=\tV \circ G$ (see Theorem \ref{bisimulation} and Definition \ref{def.hV_}). Hence, we can write $\tV=\hV \circ \Gamma$. Therefore, we can write $\partial \tV/\partial \xi=(\partial \hV/\partial z) \Gamma$, but since $\bx$ is a steady state, then $\Gamma R(\bx)=0$, and hence the first term vanishes. Similar to \eqref{e.decomp}, we can decompose the Jacobian as follows:
\[
\dot\tV(\xi) \;=\;  \sum_{\ell=1}^s \rho_\ell(x_\xi^*) \frac{\partial \tV}{\partial \xi}  Q_\ell \xi \;\le\; 0 \,,
\]
where the last inequality follows since $\tV$ is a common Lyapunov function for the systems ${\dot r}=Q_1 r,\dots,\dot r=Q_s r$.  Since $\tV$ is locally Lipschitz, the inequality holds for almost all $\xi$ using Radmacher's theorem. Next, we interpret the gradient above in the sense of the upper Dini's derivative (see Remark \ref{remark.Dini}). Then we may use Lemma 9 in \cite{MA_TAC22}, which says that if the inequality above holds almost everywhere, then it holds for all $\xi\ge 0$. This completes the proof.
\end{proof}
 
\subsubsection{The dual GLF for concentration dynamics}\
\label{sec.dualGLF}

{As we will see in section 3, nonexpansivity of the trajectories in concentration coordinates can be understood with respect to the dual GLF defined on concentration coordinates, which we will explain next.} Consider a network $\Ns$ with a steady state $\bx$. Let $V=\tV \circ R$ be defined as in Proposition \ref{lyap_dec}. Then, we claim that $V^*(x)=\hV(x-\bx)$ is a Lyapunov function for \eqref{e.ode} whenever $x(0)\in\Cs_\bx$. Positive definiteness is evident. We can verify the non-increasing property of $\tV$ by writing:
\[
\frac{d}{dt}V^*(\varphi_t(x))=\frac{d}{dt}\hV(\varphi_t(x)-\bx)=\frac{d}{dt}\tV(\psi_t(\Psi(x))\le 0,
\]
where the last inequality follows from Proposition \ref{xi_GLF}. 
 
Immediate consequences of the discussion above are as follows:

\begin{theorem}
\label{th.equiv}
Let $\Ns$ be given, and let $G:\Imgamma\to\mathbb R^\nu/\kerGamma$ be as defined in \S\ref{sec.quotient}. Then: 
\begin{enumerate}
 \item
 If $\tV$ is a $\Gamma$-translation-invariant GLF for $\Ns$, then  $\hV:=\tV \circ G$ is a dual GLF for $\Ns$.
 \item
 If $\hV$ is a dual GLF for $\Ns$, then 
 $\tV:=\hV\circ\Gamma$ is a $\Gamma$-translation-invariant GLF for $\Ns$.
\end{enumerate}
\end{theorem}

\begin{corollary}
\label{cor.dualGLF_dec}
Consider a network $\Ns=(\Ss,\Rs)$, and consider \eqref{e.ode} with a fixed admissible $R$ and a steady state $\bx$. Assume that $\hV$ is a dual GLF for $\Ns$. Then, $V(x)=\hV(x-\bx), x\in\Cs_\bx$ is {a} Lyapunov function for \eqref{e.ode} on $\Cs_\bx$.
\end{corollary}
 
\paragraph{Example (cont'd)}  Consider the BIN \eqref{e.ptm} and the introduced GLF \eqref{e.tV_ptm}. Then, the dual GLF is given as follows $\hV(z)=\|Bz\|_\infty$, where $B$ is any matrix satisfying $C=B\Gamma$, which is not unique in this case. One possible choice is the following,
\begin{equation}\label{Bptm}B=   \left[\begin{array}{rrrrrr} 1 & 0 & 0 & 0 & 0 & 0\\ 0 & 1 & 0 & 0 & 0 & 0\\ 0 & 0 & 1 & 1 & 0 & 0\\ 0 & 0 & 0 & 1 & 0 & 0\\ 0 & 0 & 0 & 1 & -1 & 0\\ 0 & 0 & 0 & 0 & 0 & 1 \end{array} \right ].
\end{equation}
Therefore, for any given steady state $\bx \in\mathbb R_+^n$, $\hV(x-\bx)$ is a Lyapunov function for \eqref{e.ode} on $\Cs_\bx$. 
  
\subsection{The class of $\ell_\infty$-norm GLFs} 

We focus now on GLFs of the form 
$\tV(r)=\|C r\|_\infty$ for some matrix 
$C=[c_1,\dots,c_m]^T \in \mathbb R^{m\times \nu}$, $m>0$. 
We will assume, without loss of generality, that $C$ is a minimal representation of $\tV$ which means that for all $k=1,\dots,m$, the set $\{r| \tV(r)=c_k^Tr\}$ has non-empty interior. 
We can make this assumption because if $C$ is not minimal, we can remove rows from $C$ to make it a minimal representation. 
  
In our previous work, we had the following characterization:

\begin{theorem}[\cite{MA_cdc14,MA_LEARN}]\label{th.metzler}
Suppose given a network $(\Ns,\Rs)$. Then, the following two statements are equivalent:
\begin{enumerate}
\item $\tV(r)=\|Cr\|_\infty$ is a GLF for $\Ns$,
\item 	$\ker C=\kerGamma$ and 
there exist Metzler matrices $\tl\Lambda_{\ell} \in \mathbb R^{2m\times 2m}, \ell=1,\dots,s$  such that 
\begin{equation}\label{e.metzler}
\tilde C Q_\ell = \tl\Lambda_{\ell} \tilde C, ~\mbox{and}~ \tl\Lambda_\ell \mathbf 1 =0, \ell=1,\dots,s, 
\end{equation}  	
where $\tilde C=[C^T,-C^T]^T$. 
\end{enumerate}  
\end{theorem}
 
In this paper, we will use an alternative form of Theorem \ref{th.metzler} which is stated below and  proved in the appendix:
 
\begin{corollary}\label{cor}
Suppose given a network $(\Ns,\Rs)$. Then, the following two statements are equivalent:
 \begin{itemize}
 \item $\tV(r)=\|Cr\|_\infty$ is  a GLF.
 \item There exists a matrix $B$ with $\rank \Gamma=\rank (B\Gamma)$ such that $C=B\Gamma$ and 
 		there exist matrices $\Lambda_{\ell} \in \mathbb R^{m\times m}, \ell=1,\dots,s$  such that 
 		\[  C Q_\ell = \Lambda_{\ell}  C ~\mbox{and}~ \mu_\infty(\Lambda_\ell) \le 0, \ell=1,\dots,s. \]
 	\end{itemize}
 \end{corollary}

 \paragraph{The dual $\ell_\infty$ GLF}  Using Corollary \ref{th.equiv}, and recalling that $C$ can be written as $C=B\Gamma$, we get the function $\hV$ as:
 \begin{equation}
 	\hV(z)=\|CG(x)\|_\infty= \|B\Gamma Gz\|_\infty= \|Bz\|_\infty,
 \end{equation}
 where the last equality follows since $\Gamma G(z)=z$ for all $z\in\Imgamma$.

 \begin{corollary}\label{th.equiv_piecewise}Let $\Ns$ be given. Then,
 	\begin{enumerate}
 		\item If  $\tV(r)=\|C r\|_\infty$ is a GLF for $\Ns$, then $\hV(z)=\|Bz\|_\infty$ is a dual GLF for $\Ns$ where $C=B\Gamma$.
 		\item If $\hV(z)=\|Bz\|_\infty$ is a dual GLF for $\Ns$, then $\tV(r)=\|B\Gamma r\|_\infty$ is a GLF for $\Ns$.
 	\end{enumerate} 
 \end{corollary}

 \mybox{\paragraph{Example (cont'd)}  Consider the BIN \eqref{e.ptm} and the introduced GLF \eqref{e.tV_ptm}. Corollary \ref{cor} asserts the existence of six {matrices} $\Lambda_\ell,\ell=1,\dots,s$.  Using the algorithms developed in \cite{PWLRj,MA_LEARN} we can find the matrices as follows:
 {\tiny \begin{align*}\Lambda_1&=\left[\begin{array}{rrrrrrrr} -1 & 0 & 0 & 0 & 0 & 0\\ 0 & -1 & 0 & 0 & 1 & 0\\ 0 & 0 & -1 & 0 & 0 & -1\\ 0 & 0 & 0 & 0 & 0 & 0\\ 0 & 0 & 0 & 0 & 0 & 0\\ 0 & 0 & 0 & 0 & 0 & 0 \end{array}\right],\Lambda_2=\left[\begin{array}{rrrrrrrr} -1 & 0 & 0 & 0 & -1 & 0\\ 0 & -1 & 0 & 0 & 0 & 0\\ 0 & 0 & -1 & 1 & 0 & 0\\ 0 & 0 & 0 & 0 & 0 & 0\\ 0 & 0 & 0 & 0 & 0 & 0\\ 0 & 0 & 0 & 0 & 0 & 0 \end{array}\right], \Lambda_3=\left[\begin{array}{rrrrrrrr} 0 & 0 & 0 & 0 & 0 & 0\\ 0 & -1 & 0 & 0 & 0 & 0\\ 0 & 0 & 0 & 0 & 0 & 0\\ 0 & 0 & 1 & -1 & 0 & 0\\ -1 & 0 & 0 & 0 & -1 & 0\\ 0 & 0 & 0 & 0 & 0 & 0 \end{array}\right],\\
 		\Lambda_4&=\left[\begin{array}{rrrrrrrr} 0 & 0 & 0 & 0 & 0 & 0\\ 0 & 0 & 0 & 0 & 0 & 0\\ 0 & -1 & -1 & 0 & 0 & 0\\ 0 & 0 & 0 & -1 & 0 & 0\\ 0 & 0 & 0 & 0 & 0 & 0\\ 0 & 0 & 0 & 0 & 1 & -1 \end{array}\right], \Lambda_5=\left[\begin{array}{rrrrrrrr} 0 & 0 & 0 & 0 & 0 & 0\\ 0 & 0 & 0 & 0 & 0 & 0\\ -1 & 0 & -1 & 0 & 0 & 0\\ 0 & 0 & 0 & -1 & 1 & 0\\ 0 & 0 & 0 & 0 & 0 & 0\\ 0 & 0 & 0 & 0 & 0 & -1 \end{array}\right], \Lambda_6=\left[\begin{array}{rrrrrrrr} -1 & 0 & -1 & 0 & 0 & 0\\ 0 & 0 & 0 & 0 & 0 & 0\\ 0 & 0 & 0 & 0 & 0 & 0\\ 0 & 0 & 0 & 0 & 0 & 0\\ 0 & 0 & 0 & 1 & -1 & 0\\ 0 & 0 & 0 & 0 & 0 & -1 \end{array}\right].\end{align*}}

 }
  
 \section{Contraction and Nonexpansivity}
 
 \subsection{Preliminaries}
 
 The are several formulations of contraction theory. We are going to present the formulation that utilizes  logarithmic norms (a.k.a{.} matrix measures).
 
 \begin{definition}[Logarithmic Norms] Let $(\mathbb R^{n},|.|)$ be a normed space, and let $\|.\|_*$ be the corresponding induced matrix norm on $\mathbb R^{n\times n}$. Then, the associated matrix measure (a.k.a{.} logarithmic norm) can be defined as follows for a matrix $A \in \mathbb R^{n \times n}$:
 	\begin{equation}\label{e.lognorm}\mu_*(A):=\lim_{h \to 0^+} \frac{ \|I+hA\|_*-1}h{.}\end{equation}
 \end{definition}
 \begin{remark}The limit \eqref{e.lognorm} exists \cite{dahlquist58}, and it  can be evaluated explicitly for the standard norms \cite{sontag14}. For instance, the following expression can be used for the $\ell_\infty$ norm:
 	\begin{equation}\label{e.lognorm_inf}\mu_\infty(A) = \max_{i} \left (a_{ii} + \sum_{j \ne i} |a_{ij}|\right ).\end{equation}
 \end{remark}

 For a dynamical system, negativity of the logarithmic norm can be linked to contraction. This result has been stated in different forms, we state the result as follows.
 \begin{theorem}[\cite{sontag14}] \label{th.contraction}Consider a dynamical system $\dot x = f(x)$ defined on a convex subset $X$ of $\mathbb R^n$. Let $| \cdot |_*$ be a norm in $\mathbb R^n$ and $\|.\|_*$ the induced matrix norm on $\mathbb R^{n \times n}$. Assume that
 \[
 \forall x \in X, \quad \mu_*\left ( \frac{\partial f}{\partial x} (x) \right )  \le c. \]
 Then for any two solutions $\varphi(t;x), \varphi(t;y)$ of the dynamical system, the following condition holds:
 \begin{equation}
 | \varphi(t;x) - \varphi(t;y) |_* \le e^{c t} | x -y|_*.
 \end{equation}
 \end{theorem}
  Note that if $c<0$  the solutions of the system are exponentially contracting. If $c=0$, then the system is nonexpansive. %
 
 \subsection{Nonexpansion in concentration coordinates}
 We state our first major result, and then develop the required theory to prove it.
  
 \begin{theorem}\label{th.mainB}
 Let $\Ns=(\Ss,\Rs)$ be a given network, and let 
 \[
 \tV(r) \;=\; \|B \Gamma r\|_\infty
 \]
 be a GLF.
Let $\bx \in \mathbb R_{\ge 0}^n$ be any positive state, and define
\[
|{x}|_B \,:=\; \|B(x-\bx)\|_\infty \,.
\]
Then, $|.|_B$ is a norm on
\[
\Cs_\bx \,:=\;(\bx+\Imgamma)\cap\mathbb R_{\ge0}^n \,,
\]
and for any two solutions $x_1(\cdot)$ and $x_2(\cdot)$ of $\dot x=\Gamma R(x)$ in $\Cs_{\bx}$, we have
\[
 |x_1(t)-x_2(t)|_B \;\leq\; |x_1(0)-x_2(0)|_B \quad \mbox{for all}\; t\geq0.
\]
In other words, the dynamics of $\Ns$  are nonexpansive in each stoichiometric class with respect to the ``$B$-norm'' in the sense defined next.

\end{theorem}

{Although the theorem assumes the existence of a GLF defined on reaction coordinates, the trajectories are nonexpansive with respect to the norm $|.|_B$ which parallels the structure of the dual GLF defined earlier in \S 2.}
Our approach to prove the theorem is to show that the appropriate logarithmic norm is non-positive. However, since Theorem \ref{th.mainB} is stated with respect to the functional $|.|_B$ on the vector space $\Imgamma$, we need to introduce the necessary notions.

\subsubsection{Norms with non-square weightings}
 	{Consider normed spaces of the form $(\mathbb R^n,|.|), n\ge 1$}. Let $B\in\mathbb R^{m\times n}$. We are interested in studying the pair $(\cV,|.|_B)$ where $\cV$ is  a vector subspace of $\mathbb R^n$ and $|.|_B$ is the weighted operator $|.|_B: z\mapsto |Bz| ${, whose domain is $\mathbb \cV \subset \mathbb R^n$}. Such operators are usually studied using the notion of Minkowski functionals or gauge functions \cite{thompson96}. However, to keep the discussion self-contained, we develop the required notions here. %

    For easier future reference, we state a simple equivalency first:
 	\begin{lemma}\label{lem.linalg}
 Let $B\in\mathbb R^{m\times n}$, let $\cV\subset \mathbb R^n$ be a vector subspace, and let $\Gamma$ be any matrix such that $\Imgamma=\cV$. Then the following are equivalent:
 \begin{enumerate}
 	\item ${\dim}(B\cV)={\dim}(\cV)$.
 	\item ${\dim}(\Im(B\Gamma))={\dim}(\Im(\Gamma))$.
 	\item $\ker(B\Gamma)=\ker(\Gamma)$.
 	\item $\ker B|_{\Imgamma}=\{0\}$.
 \end{enumerate}
 \end{lemma}
 
Next, we show next that the pair $(\cV,|.|_B)$ is indeed a normed space.
\begin{lemma}\label{lem.normB}
Let $(\mathbb R^n,|.|)$ be a normed space. Let $B\in\mathbb R^{m\times n}$, and let $\cV\subset \mathbb R^n$ be a vector subspace. If ${\dim}(B\cV)={\dim}(\cV)$, then the pair $(\cV,|.|_B)$ is a normed space.
\end{lemma}
 
\begin{proof}
Since $\cV$ is a vector subspace of $\mathbb R^n$, we immediately get that $|.|_B$ is semi-norm on  $\cV$. To show that it is a norm, we need to show that for all $x\in\cV$, $|x|_B=0$ iff $x=0$.  But this follows immediately from Lemma \ref{lem.linalg} 
 since we assumed that {${\dim}(B\cV)={\dim}(\cV)$}.  
\end{proof}

 \subsubsection{Induced matrix norms and logarithmic norms restricted to a subspace}

Let $(\mathbb R^n, |.|)$ be  a normed Euclidean space, and let $\cV \subset \mathbb R^n$ be a vector space.   
Next, we can define the $\cV$-induced matrix norm of a matrix $J \in\mathbb R^{n\times n}$ as $\|J\|_{\cV}:=\sup_{|z|=1, z\in\cV} |Jz|$, which is well-defined (finite) since the set $\{z: |z|=1, z\in\cV\}$ is compact. 
 
Similarly, we need to define the corresponding induced $(B,\cV)$-matrix norm of a matrix $J \in \mathbb R^{n\times n}$ as we did above.  %
As in Lemma \ref{lem.normB}, we assume that  ${\dim}(B\cV)={\dim}(\cV)$, i.e., the dimension of $\cV$ does not reduce if we multiply each vector by $B$ from the left. Using Lemma \ref{lem.linalg},  any vector $z\in\cV$ can be also written as $LBz$ for some matrix $L\in\mathbb R^{{n\times m}}$. Therefore, we can write:
\[ 
\sup_{\substack{|Bz|=1\\ z\in\cV }} |BJz|= \sup_{\substack{|Bz|=1\\ z\in\cV }} |BJLBz| \le \|BJL\| \sup_{\substack{|Bz|=1\\ z\in\cV }} |Bz|= \|BJL\| <\infty, 
\]
which ensures that we show that the induced $(B,\cV)$-matrix norm of any matrix $J$ is well-defined.
Therefore, we are justified in making the following definition.

\begin{definition} 
Consider the normed  space $(\mathbb R^n, |.|)$. Let $\cV \subset \mathbb R^n$ be a vector space, and let $B\in\mathbb R^{m\times n}$ be a matrix satisfying ${\dim}(B\cV)={\dim}(\cV)$. Then,
\begin{enumerate}
\item The induced $(B,\cV)$-norm of a matrix $J\in\mathbb R^{n\times n}$ restricted to $\cV$ is defined as:
\[
\|J\|_{B,\cV}:= \sup_{|z|_B=1,z\in\cV} |BJz|. 
\]
\item The $(B,\cV)$-logarithmic norm of a matrix $J\in\mathbb R^{n\times n}$ is defined as:
\[
\mu_{B,\cV}(J):= \lim_{h\to0^+} \frac{ \|I+hJ\|_{B,\cV} -1 } h. 
\] 
\end{enumerate} 
\end{definition}

\subsubsection{Computation of the logarithmic norm}

In order to prove Theorem \ref{th.mainB}, we set $\cV=\Imgamma$. Hence, we need to show that the $(B,\Imgamma)$-logarithmic norm of the Jacobian is non-positive. This is stated as follows.

\begin{lemma}
\label{mainlemmaB}
Let a network $\Ns=(\Rs,\Ss)$ be given. Fix an arbitrary $\bx \in \mathbb R_{\ge0}^n$. Assume that there exists a GLF for $\Ns$ of the form $\tV(r)=\|B\Gamma r\|_\infty$.   
Then, 
\[
\forall K \in \mathcal K_\Ns,~ \mu_{B,\Imgamma} ( \Gamma K )  \le \mu_\infty \left ( \sum_{\ell=1}^s \rho_\ell   \Lambda_{\ell} \right ) \le 0,
\]
where $\rho_1,\dots,\rho_s$ are the positive elements in $K$ as in \eqref{K_def}, and $\Lambda_1,\dots,\Lambda_s$ are constructed by solving the convex program in Corollary \ref{cor}.
\end{lemma}
 
Before proceeding to the proof of Lemma \ref{mainlemmaB}, we need to decompose matrices of the form $\Gamma K$, where $K\in\mathcal K_\Ns$ into a sum of rank-one matrices as we did in  \eqref{e.rode}. Using \eqref{e.calK}, we can write :
\begin{equation}\label{e.decomp} \Gamma K  = \sum_{\ell=1}^s \bar\rho_\ell  \Gamma  E_{j_\ell i_\ell }   = {\sum_{\ell=1}^s \bar\rho_\ell  \Gamma  e_{j_\ell} e_{i_\ell }^T }= \sum_{\ell=1}^s \bar\rho_\ell  \overbrace{(\Gamma_{j_\ell}e_{i_\ell}^T)}^{J_\ell} =  \sum_{\ell=1}^s \bar\rho_\ell J_\ell,\end{equation}
where $\Gamma_1,\dots,\Gamma_n$ are the columns of $\Gamma$.  In addition, we need the following basic lemma:
\begin{lemma}\label{simplelemma} 
Given a network $\Ns$. Let $J_\ell, Q_\ell, {\ell=1,\ldots,s}$ be defined as above. Then:
\begin{enumerate}
 \item $\forall \ell, \Gamma Q_\ell = J_\ell \Gamma$.
 \item Let $C,B,\Lambda_\ell,\ell=1,\dots,s$ be defined as in Corollary \ref{cor}, and let $D$ be {a} matrix whose rows form a basis for the left null space of $\Gamma$.  {T}hen there exist matrices $Y_1,\dots,Y_s$ such that $BJ_\ell=\Lambda_\ell B+Y_\ell D$, $\ell=1,{\ldots},s$. 
\end{enumerate}
\end{lemma}
\begin{proof}
The first statement follows from these equalities:
\[
\Gamma Q_\ell = \Gamma e_{j_\ell} \gamma_{i_\ell}^T = \Gamma_{j_\ell} \gamma_{i_\ell}^T= \Gamma_{j_\ell} e_{i_\ell}^T \Gamma = J_\ell \Gamma \,.
\]
Now, using Corollary \ref{cor}, we have $CQ_\ell = \Lambda_\ell C$, we get:
\[
B\Gamma Q_\ell = \Lambda_\ell B \Gamma \Longrightarrow B J_\ell \Gamma = \Lambda_\ell B \Gamma \Longrightarrow (BJ_\ell - \Lambda_\ell B) \Gamma =0.\]
Therefore, the rows of $BJ_\ell - \Lambda_\ell B$ belong to the left null space of $\Gamma$. Hence, there exist matrices $Y_1,\dots,Y_s$ such that $BJ_\ell-\Lambda_\ell B=Y_\ell D$, $\ell=1,\dots,s$. By rearranging we get the required expression.
\end{proof}
 
\paragraph{Proof of Lemma \ref{mainlemmaB}}

Using the previous definitions, we are interested in bounding the following expression:
\[
\mu_{B,\Imgamma}(\Gamma K) := \limsup_{h\to 0^+} \frac{ \|I+h \Gamma K\|_{B,\Imgamma} -1 } h, 
\]
where $K\in \mathcal K_\Ns$.
We proceed as follows:
\begin{align*}
 \left \|   I + h \Gamma K    \right \|_{B,\Imgamma} & = \sup_{\substack{\|Bz\|_\infty =1 \\ z\in\Imgamma }} \left \| B \left ( I + h\Gamma K  \right ) z \right \|_\infty
 \mathop{=}^{(\star)}   \sup_{\substack{\|Bz\|_\infty =1 \\ z\in\Imgamma }} \left \| Bz + h   \sum_{\ell=1}^s \rho_\ell B J_\ell z   \right \|_\infty
 \\& \mathop{=}^{(\clubsuit)}   \sup_{\substack{\|Bz\|_\infty =1 \\ z\in\Imgamma }} \left \| Bz + h   \sum_{\ell=1}^s \rho_\ell(\Lambda_\ell B+Y_\ell D)z   \right \|_\infty \mathop{=}^{(\spadesuit)}   \sup_{\substack{\|Bz\|_\infty =1 \\ z\in\Imgamma }} \left \| \left ( I + h   \sum_{\ell=1}^s \rho_\ell \Lambda_\ell\right ) Bz   \right \|_\infty \\
 \\ &  \le      \sup_{\substack{\|Bz\|_\infty =1 \\ z\in\Imgamma }} \left \| I + h   \sum_{\ell=1}^s \rho_\ell   \Lambda_{\ell}     \right \|_\infty \|B z \|_\infty = \left \| I + h   \sum_{\ell=1}^s \rho_\ell   \Lambda_{\ell}     \right \|_\infty,
\end{align*}
where the equality $(\star)$ follows from \eqref{e.decomp}, equality $(\clubsuit)$ follows from Lemma \ref{simplelemma}, and equality $(\spadesuit)$ follows since $Dz=0$ for all $z\in\Imgamma$.
 
Therefore, the expression of the logarithmic norm above can be written as:
\begin{equation}\label{e.lognorm_ineq2}  \mu_{B,\Imgamma}(\Gamma K)  \le \mu_\infty \left ( \sum_{\ell=1}^s \rho_\ell   \Lambda_{\ell} \right ) \le   \sum_{\ell=1}^s \rho_\ell  \mu_\infty( \Lambda_{\ell}) = 0, 
\end{equation}
where the inequalities follow by the subadditivity of the logarithmic norm and Corollary \ref{cor}.  \hfill $\square$
 
\paragraph{Proof of Theorem \ref{th.mainB}} 

This follows immediately from combining Theorem \ref{th.contraction} with Lemma \ref{mainlemmaB} since the Jacobian of \eqref{e.ode} can always be written as $\Gamma K$ for some $K\in\mathcal K_\Ns$ as in \eqref{e.dR}.  \hfill $\square$

\subsubsection{Boundedness}
 
A BIN that is not conservative is not guaranteed to have bounded trajectories. 
In fact, nonexpansivity does not preclude unboundedness. Nevertheless, if a network admits a steady state then we get boundedness as the following corollary shows. 
 
\begin{definition}
Let $\Ns$ be a given a network, and let the dynamical system \eqref{e.ode} be given with a fixed kinetics $R$. Fix a stoichiometric class $\Cs$. Then, the dynamical system \eqref{e.ode} is said to be uniformly bounded over $\Cs$ if for each compact set $K \subset \Cs$, there exists a compact set $\tl K$ such that $\varphi(t;K) \subset \tl K$ for all $t\ge 0$. 
\end{definition}
 
If a GLF exists, then checking uniform boundedness is easy as the following corollary shows.
\begin{corollary}\label{cor.boundedness}
Let $\Ns=(\Ss,\Rs)$ be a network that admits a GLF $\tV(r)=\|B\Gamma r\|_\infty$. Consider \eqref{e.ode} with a given kinetics $R$. Then, if a steady state $\bx$ exists, the dynamical system \eqref{e.ode} is uniformly bounded over $\Cs_\bx$. 
\end{corollary}

\begin{proof}
Denote 
\begin{equation}\label{e.K_tilde} 
\tl K_M:=\{x :  |x-\bx|_B \le M \} \cap \Cs_\bx. \end{equation} 
Pick any nonempty compact set $K\subset \Cs_\bx$. Then, let $M$ be the minimal number such that $K\subset\tl K_M$. Then, 
we apply Theorem \ref{th.mainB} by comparing an arbitrary trajectory $x_1(t)$ starting in $\tl K_M$ with $x_2(t)\equiv \bx$. By nonexpansiveness, $x_1(t) \in \tl K_M$ for all $t\ge0 $.
\end{proof}

\subsubsection{Remark on global convergence to steady states} \label{s.global}

If a network admits a GLF,  global stability can be checked by a graphical LaSalle argument \cite{PWLRj}. Alternatively,	it has been shown \cite{blanchini17} that    the existence of a non-degenerate positive steady state  implies global asymptotic stability. Non-degeneracy of the Jacobian can be tested easily as shown in \cite{blanchini14,MA_LEARN}, and is implemented in the computational package \texttt{LEARN}.  However, Corollary \ref{cor.boundedness} does not assume non-degeneracy to show boundedness. 
 
Recent work in contraction theory \cite{jafarpour21,duvall24} has also provided results on the question of convergence under the assumption of {nonexpansivity}. As has been stated originally in \cite{jafarpour21} and then in \cite[Theorem 2]{duvall24}, nonexpansivity with respect to a polyhe{dr}al norm coupled with analyticity of the vector field (which applies if all reaction rates are real-analytic functions, as is the case with mass-action and most other kinetics) implies that each trajectory is either unbounded or converges to the set of steady states.

Combining the latter result with Corollary \ref{cor.boundedness} means that the existence of a single steady state implies bound{ed}ness, and hence convergence to the set of steady states.  However, this result does not guarantee the uniqueness of a steady state in a stoichiometric class. For that and to get global stability, we still need robust non-degeneracy of the Jacobian, see for instance \cite[Theorem 19,iii]{jafarpour21}.

 \subsubsection{Alternative statement of nonexpansivity}
 The results can be stated explicitly in the stoichiometric class by using a transformation. We provide the details below for completeness.  Recall that $D$ is a matrix whose rows form a basis for the left null space of $\Gamma$. Hence, let  $S=[S_1^T,D^T]^T$ where $S_1$ is any matrix chosen so $S$ is invertible.  Consider \eqref{e.ode} with a fixed vector $\bx$. We will use the following coordinate change:   \[ \tl z= S(x-\bx)=\begin{bmatrix} S_1 (x-\bx) \\ D(x-\bx) \end{bmatrix}=:\begin{bmatrix}\tl z_1 \\ \tl z_2 \end{bmatrix}. \] 
 The dynamics of \eqref{e.ode} can be written as:
 \[\dot{\tl z}= \begin{bmatrix} S_1 \\ D \end{bmatrix} 
 \Gamma R(x)= \begin{bmatrix} S_1  \Gamma R(S^{-1} \tl z+\bx) \\ 0 \end{bmatrix}, \]
 where the second entry is zero since $D\Gamma=0$. 
 
 Since we have $x(0)\in \Cs_{\bx}$, then $x(t)-{\bx} \in \Imgamma$ for all $t$, and we immediately get that $\tl z_2(t) = D(x-{\bx})=0$ for all $t$. Furthermore, let us write $S^{-1}=: [ U_1,U_2]$, where $U_1 \in \mathbb R^{n\times(n-d)},{U}_2\in\mathbb R^{n\times d}$, $d=\dim(\ker(\Gamma^T))$.
 
 Therefore,
 
 \[S^{-1} \tl z + \bx = [U_1,U_2]\begin{bmatrix}\tl z_1 \\ 0 \end{bmatrix} + \bx = U_1 \tl z_1 + \bx .\]
 
 Finally we get the dynamics on the $n-d$ manifold as:
 \begin{equation}\label{e.red_Sys} { \dot{\tilde{z}}_1 = S_1 \Gamma R(U_1 \tl z_1 + \bx ), ~\tl z_1(0)~\mbox{satisfies}~U_1 \tl z_1(0) + \bx\ge0 .}\end{equation}
 In the new coordinates, the problem reduces to checking contraction of \eqref{e.red_Sys} with respect to the norm $\|.\|_{BU_1}: \tl z_1 \mapsto \|B U_1 \tl z_1 \|_{\infty}$.  Hence, we state the following corollary.
 \begin{corollary}
 Let a network $\Ns$ be given.	Let $S,S_1,U_1$ be matrices defined as above. Consider the ODE \eqref{e.red_Sys} whose Jacobian is $S_1  \Gamma \frac{\partial R}{\partial x} U_1$. Consider the norm $\|.\|_{B{U}_1}: \hat z_1 \mapsto \|BU_1\hat z_1\|_\infty$ and the associated logarithmic norm $\mu_{B{U}_1}$. Then, $\mu_{B{U}_1}(S_1 \Gamma \frac{\partial R}{\partial x}  U_1 )=\mu_{B,\Imgamma}(  \Gamma  \frac{\partial R}{\partial x}  ) \le 0$. In other words, the logarithmic norm of the reduced system is independent of the choice of $S_1$ and is non-positive. In addition,  consider any two solutions $\hat\varphi(t;{\hat z}_{1a}),\hat\varphi(t;\hat z_{1b})$ of \eqref{e.red_Sys}, then the following condition holds:
 \[ |\hat\varphi(t;\hat z_{1a})-\hat\varphi(t;\hat z_{1{b}})|_{B{U}_1} \le |\hat z_{1a}-\hat z_{1b}|_{B{U}_1},~\mbox{for all}~t\ge0.\]
 \end{corollary}

 \subsection{Nonexpansion of the extent-of-reaction system}
 We can perform  a similar contraction analysis for the extent-of-reaction system, but thanks to Theorem \ref{bisimulation}, we can use Theorem \ref{th.mainB} to show the following immediately:
 \begin{theorem}\label{th.mainC}
 Let $\Ns=(\Ss,\Rs)$ be a given network, and let {$\tV(r) = \|C r\|_\infty$  be a GLF.}
 Let $\bx \in \mathbb R_{\ge 0}^n$ be any positive state, and define $|{\xi}|_C := \|C\xi\|_\infty$.
 Then,  
 for any two solutions $\xi_1(\cdot)$ and $\xi_2(\cdot)$ of $\dot\xi=R(\bx+\Gamma \xi)$, we have
 \[
 |\xi_1(t)-\xi_2(t)|_C \;\leq\; |\xi_1(0)-\xi_2(0)|_C \quad \mbox{for all}\; t\geq0.
 \]
\end{theorem}
\begin{proof}
With reference to the commutative diagram in Theorem \ref{bisimulation} we have $x_1(t)=\Gamma \xi_1(t) + \bx $, and $x_2(t)=\Gamma \xi_2(t) + \bx$. Note that $x_1(t),x_2(t)\in\Cs_\bx$. Then, we can write the following:
\begin{align*}
|\xi_1(t)-\xi_2(t) |_C &= \| C ( \Psi(x_2(t))-\Psi(x_1(t))  ) \|_\infty = \| B\Gamma ( G(x_2(t)-\bx) - G(x_1(t)-\bx)  ) \|_\infty \\ 
& = \| B ( x_2(t)-\bx -x_1(t) + \bx ) \|_\infty = \| B(x_2(t)-x_1(t) ) \|_\infty \\
& \le \| B(x_2(0) - x_1(0))\|_\infty = \| B\Gamma ( \xi_2(0)- \xi_1(0)) \|_\infty = |x_2(0)-x_1(0)|_C,
\end{align*} 
where the first equality uses the commutativity diagram in Theorem \ref{bisimulation}, while the second equality uses \eqref{Psi}, the third equality uses the property that $\Gamma G(z)=z$ for all $z\in \Imgamma$, and the inequality follows from Theorem \ref{th.mainB}.
\end{proof}

\subsection{Examples and discussion}

\mybox{\subsubsection{The PTM cycle (cont'd)}
Let us consider again the PTM cycle \eqref{e.ptm}.  Using Theorem \ref{th.mainB} the network is nonexpansive with respect to the norm $|.|_B$ in each stoichiometric class, where $B$ is given by \eqref{Bptm}. 
Figure \ref{f.ptm_timetraj} shows a numerical simulation of the distance between pairs of trajectories  where it can be seen that they are approaching each other. In the next section, we will show that the system is strictly contractive over any compact set in the positive orthant. }

\begin{figure}
\centering

\subfigure[]{\includegraphics[width=0.496\textwidth]{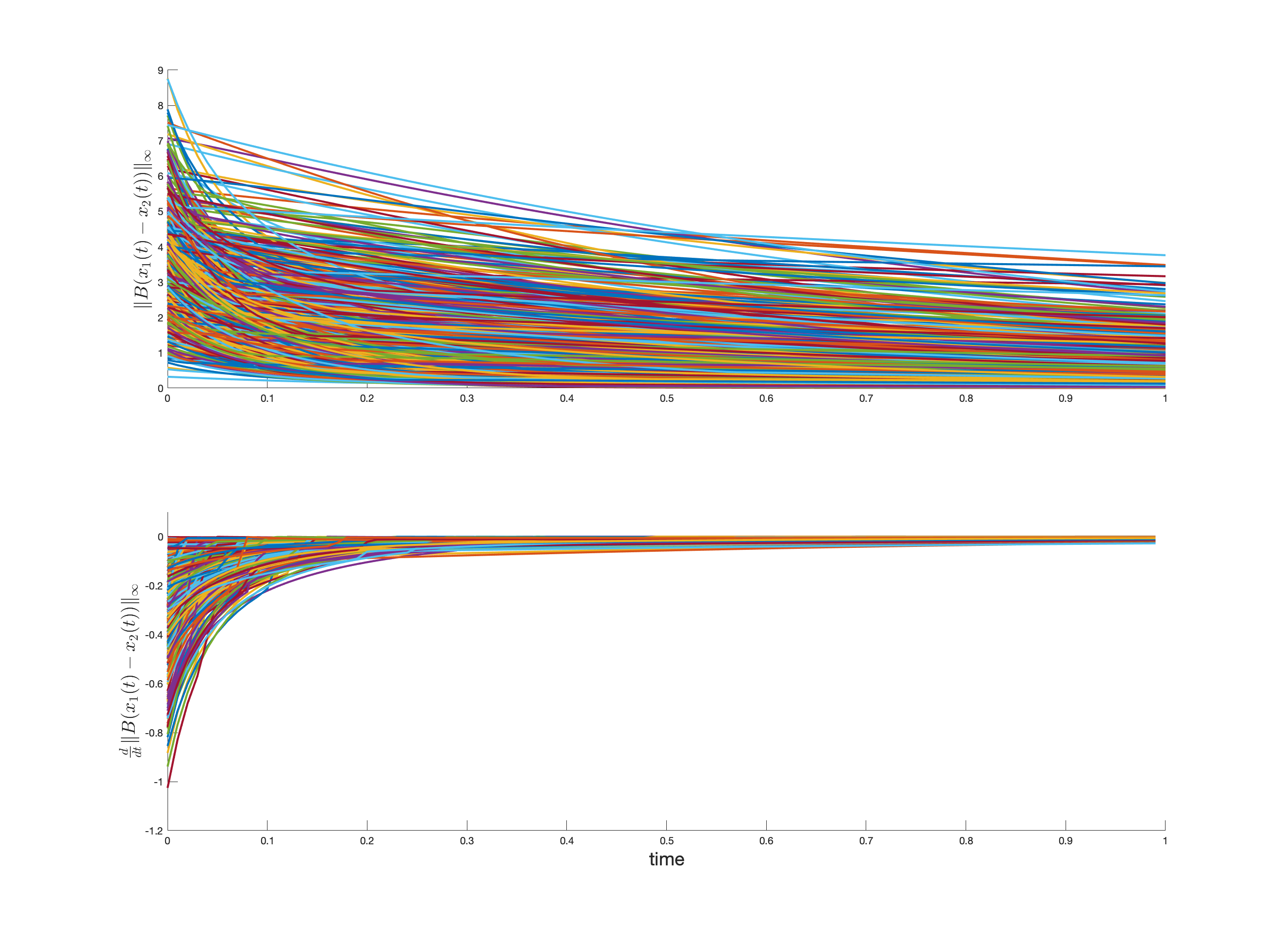}}
\subfigure[]{\includegraphics[width=0.496\textwidth]{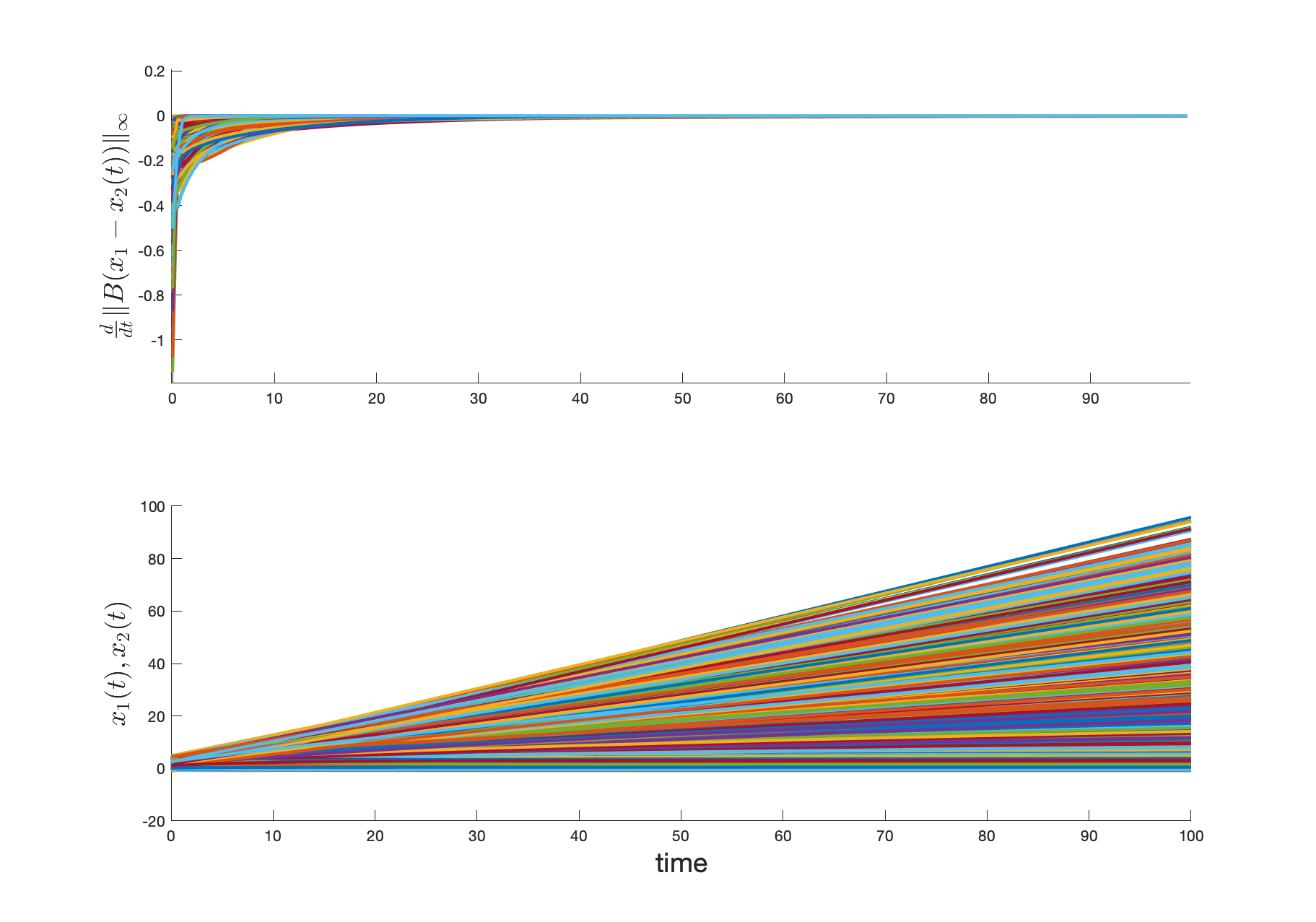}}
\caption{Nonexpansivity for two examples (a) The $B$-distance between 500 pairs of randomly generated trajectories for the PTM system with $R(x)=[se, c_1,pd , c_2]^T$, and its numerically-calculated time-derivative using MATLAB's \texttt{diff} subroutine. (b) An unbounded nonexpansive system. The time-derivative of the $B$-distance between 500 pairs of randomly generated trajectories for the  system \eqref{e.unstable} with $R(x)=[c,1,ab]^T$, and a plot of the unbounded trajectories. }
\label{f.ptm_timetraj}
\end{figure}

\mybox{\subsubsection{An unstable system}
{N}onexpansivity does not imply boundedness. For example, let us consider an example studied in \cite{PWLRj}:
\begin{equation} \label{e.unstable} C \longrightarrow A, ~ \emptyset \longrightarrow B, ~ A+B \longrightarrow C.\end{equation}
This system admits the GLF $\tV(r)=\|B\Gamma r\|_\infty$, where 
\[B=\begin{bmatrix} 1 & 0 & 0 \\ 0 & 1 & 0 \\ 1 & -1 & 0 \end{bmatrix}. \]
Note, however, that the trajectories of the system can become unbounded when a steady state fails to exists. Figure \ref{f.ptm_timetraj}(b) shows unbounded trajectories. Nevertheless, the distance between the trajectories never increases.}

\section{Strict contraction over positive compact sets}

In the previous section we have shown that any network admitting a GLF is nonexpansive. In this section, we will strengthen this to show strict contraction over arbitrary compact sets subject to additional computationally-verifiable conditions. 

\subsection{Background: siphons and persistence}
We revisit the concept of siphons, which will assist us in establishing strict contraction over compact sets.
A dynamical system defined over the positive orthant is said to be persistent if trajectories that start in the interior of the positive orthant do not approach the boundary asymptotically \cite{angeli07p}. This is also known as non-extinction since it is tantamount to asking that {the species that} are initially present in the BIN do not go extinct. 

An approach that can be used to establish persistence is based on the theory of siphons \cite{angeli07p}. 

\begin{definition} \label{d.trivialSiphon} Let $\Ns=(\Ss,\Rs)$ be a network. A set $P\subset\mathscr S$ is said to be a siphon if {it} satisfies the following statement:
	$\forall X \in P$, if $X$ is a product of some reaction $\R_j$, then $\exists X'\in P$ such that $X'$ is a reactant of $\R_j$. In addition,
	a siphon is \emph{trivial} if it contains {the} support of a conservation law, and it is \textit{critical} otherwise.
\end{definition}

If a network contains only trivial siphons, then its trajectories that start in the positive orthant do not approach the boundary asymptotically for any choice of kinetics, which is known as structural persistence \cite{angeli07p}. In fact, a slightly stronger statement can be established. If we choose all our initial conditions from a compact set $K\subset \mathbb R_+^n$, then the corresponding trajectories are uniformly separated from the boundary. This is shown in the following result which is a modification of a corresponding result in \cite{MA_TAC22}:
\begin{lemma}
	\label{eventuallyK}
	Consider a network $\Ns$. Consider the corresponding dynamical system \eqref{e.ode} with given admissible kinetics $R$, and assume that it is uniformly bounded over a given a stoichiometric class $\Cs$. 
	Assume that all siphons are trivial.
	Then, 
	for any compact $K \subset \Cs \cap (0,+\infty)^n$, there exist a number $\ve>0$ and
	a convex and compact $\tilde{K}$  in $[\ve,+ \infty)^n$ such that
	$\varphi (t;K) {\subset} \tilde{K}$ for all $t \geq 0$.  
\end{lemma}
\begin{proof}
	Using \cite[Lemma 5]{MA_TAC22},  there exist $\ve>0$ and
	$\tilde{K_1}$ compact in $[\ve,+ \infty)^n$ such that
	$\varphi (t;K) {\subset} \tilde{K}_1$ for all $t \geq 0$.
	Our desired \emph{convex} $\wt K$ is any compact and convex set that satisfies $\wt K_1\subset \wt K \subset [\ve,\infty)^n$. One such set is the closed box $[\ve,x_1^*]\times \dots \times [\ve,x_n^*]$, where $x^* \in\wt K_1$ is the supremum point of $\wt K_1$.
\end{proof}

\begin{remark}
	Using Corollary \ref{cor.boundedness}, uniform boundedness follows for networks admitting a GLF whenever a steady state exists.
\end{remark}

\subsection{Strict contraction over positive compact sets} \label{s.strictcontraction}
Recall that the upper bound in \eqref{e.lognorm_ineq2} is $\mu_\infty(\sum_{\ell=1}^s \rho_\ell \Lambda_{\ell})$ which is written in term{s} of the partial derivatives of $R$. Therefore, {if we} can establish a \emph{lower bound} on the partial derivatives of $R$, then we can hope to show strict contraction. However, for networks that contain bimolecular reactions, this is not generally possible since the partial derivatives of $R$ can get arbitrarily small when we get close to the boundary. Hence, our technique will be to establish contraction with respect to \emph{compact sets}. If we assume the absence of trivial siphons, then Lemma \ref{eventuallyK} assures us that there is a uniform positive lower bound on all the states. We will show that this can be applied to few examples and then we state the general result.

\paragraph{Example 1: Three-body binding}  Unlike standard two-body binding, three-body binding \cite{douglass13} requires to a bridging molecule $B$ for the formation of a trimeric complex $\rm ABC$. The {BIN} can be written as follows:
\begin{align} \label{three_body}
	\rm A+ \rm B &\xrightleftharpoons[R_5]{R_1} {\rm AB}, ~ { \rm C+B} \xrightleftharpoons[R_6]{R_2} {\rm BC},\\
	{\rm A+BC } & \xrightleftharpoons[R_7]{R_3} {\rm ABC} \xrightleftharpoons[R_4]{R_8} {\rm C+AB}.
\end{align}
The network admits the GLF $\tV(r)=\|\Gamma r\|_\infty$, so $B=I$. The certificate matrices $\Lambda_\ell, \ell=1,\dots,12$ can be computed via Corollary \ref{cor}, and hence  the upper bound \eqref{e.lognorm_ineq2} can be computed explicitly, the details are provided in \S 6.2. For our purposes here,  \eqref{e.lognorm_ineq2} gives us:
\[\mu_{\scriptscriptstyle %
\infty,\Imgamma}(J)\le -\min\{\rho_{1}+\rho_{2},\rho_{3}+\rho_{4},\rho_{5}+\rho_{6}, \rho_{7}+\rho_{8},\rho_{9}+\rho_{10},\rho_{11}+\rho_{12}\}:=\rho(x) <0, \]
where the last inequality holds for all $x \in \mathbb R_+^n$ by assumption AK3.  Therefore, this establishes that the Jacobian has a strictly negative logarithmic norm over the positive orthant. However, this falls short of proving strict contraction.  Nevertheless, the situation can be remedied for compact sets by using Lemma \ref{eventuallyK}. So, let $K\subset \mathbb R_+^{n}$ be any positive compact set. By Lemma \ref{eventuallyK} there exists an $\ve>0$ and a convex and compact $\wt K$ such that $\varphi(t;K){\subset} \wt K \subset [\ve,\infty)^n$ for all $t\ge0$.
So let $c :=\min_{x\in\wt K} \rho(x) >0$. So if $x(0),y(0) \in K \subset \Cs_\bx$, then we get strict contraction
\[
|{x(t)-y(t)}|_\infty \;\leq\; e^{-ct} |{x(0)-y(0)}|_\infty \; \mbox{for all}\; t\geq0.
\]

\paragraph{Example 2: PTM cycle (continued)} 

Let us continue considering the PTM cycle. In order to bound the logarithmic norm using \eqref{e.lognorm_ineq2}, we need to examine the expression $\mu_\infty(\sum_{\ell=1}^s \rho_\ell(x) \Lambda_{\ell})$. By substituting the values of $\Lambda_\ell,\ell=1,\dots,6$ we get:  (the dependence on $x$ is dropped for notational simplicity)
\[  
\left[\begin{array}{cccccc} -\rho_{1}-\rho_{2}-\rho_{6} & 0 & -\rho_{6} & 0 & -\rho_{2} & 0\\ 0 & -\rho_{1}-\rho_{2}-\rho_{3} & 0 & 0 & \rho_{1} & 0\\ -\rho_{5} & -\rho_{4} & -\rho_{1}-\rho_{2}-\rho_{4}-\rho_{5} & \rho_{2} & 0 & -\rho_{1}\\ 0 & 0 & \rho_{3} & -\rho_{3}-\rho_{4}-\rho_{5} & \rho_{5} & 0\\ -\rho_{3} & 0 & 0 & \rho_{6} & -\rho_{3}-\rho_{6} & 0\\ 0 & 0 & 0 & 0 & \rho_{4} & -\rho_{4}-\rho_{5}-\rho_{6} \end{array}\right].
\]

Let $J$ be a square matrix.
Recall that $\mu_\infty(J)=\max_i \sigma_i(J)$, where $\sigma_i(J)=[J]_{ii}+\sum_{j\ne i} |{[}J_{ij}{]}|$. For the matrix above, note that $\sigma_{3}(\sum_{\ell=1}^s \rho_\ell \Lambda_\ell)$ and  $\sigma_{5}(\sum_{\ell=1}^s \rho_\ell \Lambda_\ell)$ are both \emph{identically} zero. Hence,  $\mu_\infty(\sum_{\ell=1}^s \rho_\ell \Lambda_{\ell})=0$ for all possible choices of $\rho_1,\dots,\rho_6$. This situation precludes using an argument similar to the previous example. Nevertheless, we can adapt a trick used in \cite{margaliot16} by using a non-singular diagonal matrix $P$ to define a scaled norm $|.|_{PB}$.  Hence, the equation $B\Gamma (\sum_{\ell} \rho_\ell Q_\ell ) = (\sum_{\ell} \rho_\ell \Lambda_\ell ) B \Gamma $, can be written as 
\begin{equation}
	PB\Gamma \left (\sum_{\ell} \rho_\ell Q_\ell \right ) = \left (\sum_{\ell} \rho_\ell P\Lambda_\ell P^{-1} \right )  B \Gamma.
\end{equation}
Hence, we can evaluate $\mu_\infty(\sum_{\ell=1}^s \rho_\ell P\Lambda_{\ell} P^{-1})$ for the scaled norm.
In our case,  let $P=\diag([1+\theta,1+\theta,1,1+\theta,1,1+\theta])$ for some small $\theta>0$. 
Then, we get
\begin{align*} \mu_\infty\left(\sum_{\ell=1}^s \rho_\ell P\Lambda_{\ell} P^{-1}\right)& = -\min\{\rho_1-\theta (\rho_2+\rho_6), \rho_3+\rho_2-\theta\rho_1, (\rho_1+\rho_2+\rho_4+\rho_5) \tfrac{\theta}{1+\theta}, \\ & \qquad \qquad \quad \rho_4- \theta(\rho_3+\rho_5), (\rho_3+\rho_6)\tfrac{\theta}{1+\theta}, \rho_5+\rho_6-\theta \rho_4   \}.\end{align*}

We are going to show that it is always possible to choose $\theta$ such that the expression above is negative over an appropriate set. So let $K$ be a positive compact set.  By Lemma \ref{eventuallyK} there exists an $\ve>0$ and a convex and compact $\wt K$ such that $\varphi(t;K){\subset} \wt K \subset [\ve,\infty)^n$ for all $t\ge0$.  Define the following constant:
\[ \bar\theta=\min_{x\in\wt K}\left\{ \frac{\rho_1}{\rho_2+\rho_6}, \frac{\rho_{3}+\rho_2}{\rho_1}, \frac{\rho_4}{\rho_3+\rho_5}, \frac{\rho_5+\rho_6}{\rho_4}\right\}>0, \]
which exists and is well-defined due to the positivity and compactness of $\wt K$. Therefore, for any choice of $\theta \in (0,\bar\theta)$, the logarithmic norm  $\mu_\infty\left(\sum_{\ell=1}^s \rho_\ell P\Lambda_{\ell} P^{-1}\right)$ is negative over $\wt K$, and is upper bounded by $-c$ for some $c>0$. Therefore,   if $x(0),y(0) \in K \subset \Cs_\bx$, then we get strict contraction
\[
|{x(t)-y(t)}|_{PB} \;\leq\; e^{-ct} |{x(0)-y(0)}|_{PB} \; \mbox{for all}\; t\geq0.
\]

\subsubsection{Weakly contractive networks: a structural notion} \label{s.weakcontraction}
\paragraph{Weakly contractive matrices} As seen in the last example, a matrix whose logarithmic norm is identically zero can be scaled using a diagonal transformation into a matrix with a \emph{negative} logarithmic norm. Here, we generalize the concept introduced in \cite{margaliot14} using the following definition:
\begin{definition}\label{def.weakcontract}
Let ${\Lambda} \in \mathbb R^{n\times n}$ be {a} matrix that satisfies $\sigma_i(\Lambda)\le 0, i=1,\dots,n$. Let
\[
S_-:= \{ i \in{1,\dots,n} | \sigma_i(\Lambda)<0  \}
\]
and
\[
S_0:= \{ i \in{1,\dots,n} | \sigma_i(\Lambda)=0  \} \,.
\]
If for each $i\in S_0$ there exists a series of indices $i,i_0,\dots,i_{k_i-1} \in S_0$ and $i_{k_i} \in S_-$ such that ${\Lambda}_{ii_0},{\Lambda}_{i_0i_1},\dots,{\Lambda}_{i_{k_i-1}i_{k_i}}>0$, then $\Lambda$ is said to be \emph{weakly contractive}. In addition, we call $k_i$ the nonexpansivity depth of index $i$. Then, let $S_{0k} \subset S_0$ be all the indices that have a nonexpansivity depth exactly $k$. Finally, $k:=\max_{i \in S_0} k_i$ is the maximal nonexpansivity depth of $\Lambda$.  If $S_0 = \emptyset$, then $k:=0$.
\end{definition}

Let $\theta>0$ and a weakly contractive $\Lambda$ be given. Then we define the \emph{contractor matrix} $P_\theta:=\prod_{{j}=1}^k P^{(j)}$, where $k$ is the maximal nonexpansivity depth of $\Lambda$. The matrices  $P^{({\lambda})}$ are diagonal matrices that are defined as follows. For ${j}=1$, then $p_{ii}^{(1)}=1+\theta$ if $i\in S_{-}$, and $p_{ii}^{(1)}=1$ otherwise.  Next, assume we have defined the matrices $P^{(1)},\dots,P^{(j-1)}$. To define $P^{(j)}$, we write $p_{ii}^{(j)}=1+\theta$ if $i\in S_{-} \cup S_{01} \cup \dots \cup S_{0(j-1)}$, and $p_{ii}^{(j)}=1$ otherwise.

Next, we have the following lemma:
\begin{lemma}\label{lem.contractor}
	Let $\Lambda \in \mathbb R^{n \times n}$ be weakly contractive with maximal nonexpansivity depth $k$. Then, $\exists \bar\theta>0$ such that  for all $ \theta \in (0,\bar\theta)$ there exists $c_\theta<0$ such that the corresponding  contractor matrix $P_\theta$ defined above satisfies $\mu_{\infty}(P_\theta JP_\theta^{-1})<c_\theta <0$.
	\end{lemma}
	\begin{proof} We will prove the following statement first: Let $P^*:=\prod_{j=1}^{\tilde k} P^{(j)}$ for some $\tilde k\le k$. Then $\exists c_\theta<0$ such that $\sigma_i(P^* \Lambda {P^*}^{-1})<c_\theta$ for $i \in S_- \cup S_{0j}, j \le \tilde k$ and $\sigma_i(P^* \Lambda {P^*}^{-1})=0$ otherwise. We start with $\tilde k=1$. Let $ P^*=P^{(1)}$.  Let $i \in S_{01}$, then  
		\[\sigma_i(P^* \Lambda {P^*}^{-1})= \Lambda_{ii}+ \sum_{j\ne i, j \not\in S_- } |\Lambda_{ij} |+ \sum_{j\in S_- } \frac 1{1+\theta }|\Lambda_{ij} |< \frac{-\theta}{1+\theta}  \sum_{j\in S_- } |\Lambda_{ij} | <0 \]
		for any choice {of} $\theta>0$. 
		 If $i\in S_{0j}, 1<j\le k$, then $\sigma_i(P^* \Lambda {P^*}^{-1})=0$.	If $i \in S_{-}$, then 
				\[\sigma_i(P^* \Lambda {P^*}^{-1})= \Lambda_{ii}+ \sum_{ j \in S_{0} } (1+\theta)   |\Lambda_{ij} |+ \sum_{j\ne i, j\in S_- } |\Lambda_{ij} |=  \underbrace{\sigma_i(\Lambda)}_{<0}+ \sum_{ j \in S_{0} } \theta   |\Lambda_{ij} | <0.\]
				Note exists a small enough $\bar\theta>0$ such that the last inequality holds  for all $i\in S_-$ and for all $\theta \in (0,\bar\theta)$. 
				Then, we define  $ c_\theta:= \max _{i \in S_-} \left (\sigma_i(\Lambda) + \sum_{ j \in S_{0} } \theta   |\Lambda_{ij} |\right)$.
				
			 For the induction step, assume that the statement holds for $\tilde k-1$. Then, the same argument used above can be applied to $\tilde \Lambda=\prod_{j=1}^{\tilde k-1} P^{(j)}  \Lambda \left (\prod_{j=1}^{\tilde k-1} P^{(j)}\right )^{-1} $ where it can be immediately seen that $\tilde S_-=S_- \cup S_{01} \cup \dots \cup S_{0(\tilde k-1)}$ and $S_0= S_{0(\tilde k)} \cup \dots \cup S_{0k}$.
		\end{proof}
	
\paragraph{Example} Consider the matrix 
\[ \Lambda =     \left[\begin{array}{rrrrr} -1 & 1 & 0 & 0 & 0\\ 0 & -1 & 1 & 0 & 0\\ 0 & 0 & -1 & 1 & 0\\ 0 & 0 & 1 & -2 & 0\\ 0 & 0 & 0 & 1 & -1 \end{array}\right].
 \] Note that $\mu_\infty(\Lambda)=0$. Then $S_0=\{1,2,3,5\}, S_-=\{4\}$. In addition, $\Lambda$ is weakly contractive since all indices are connected to index 4. Furthermore, $S_{01}=\{3,5\},S_{02}=\{2\},S_{03}=\{1\}$.  The corresponding contractor matrix can be written as follows:
 \[P^{{\theta}}=\begin{bmatrix} 1 &0 & 0 & 0 & 0 \\ 0 &1+\theta & 0 & 0 & 0 \\   0 & 0&(1+\theta)^2 & 0 & 0    \\   0 & 0 & 0&(1+\theta)^3  & 0  \\   0 & 0 & 0& 0 & (1+\theta)^2    \end{bmatrix}. \]
 Then, Lemma \ref{lem.contractor}  implies that $\exists~\bar\theta>0$ such that  $\forall \theta \in (0,\bar \theta) \exists c_\theta>0$, where $\mu_\infty(P^{{\theta}}J(P^{{\theta}})^{-1})<c_\theta<0$. For instance, $\bar\theta=1$ for the example above. Choosing $\theta=0.1$, for instance, gives $\mu_\infty(P^{{\theta}}J(P^{{\theta}})^{-1})=  -0.0909 <0$. \hfill $\blacksquare$
 
Since we want to apply the above definition to matrices of the form $\sum_\ell \rho_\ell \Lambda_\ell$, we state the following basic Lemma which follows immediately from the subadditivity and homogeneity of the logarithmic norm.
\begin{lemma} Let {the} square matrices $\Lambda_\ell,\ell=1,\dots,s$ be given such that $\mu_\infty(\Lambda_\ell)\le0$, and denote $\bar\Lambda(\rho_1,\dots,\rho_s)=\sum_\ell \rho_\ell \Lambda_\ell$. Then, $\bar\Lambda(\rho_1,\dots,\rho_s)$ is weakly contractive for all $\rho_1,\dots,\rho_s>0$ iff  ${\bar\Lambda}(\rho_1,\dots,\rho_s)$ is weakly contractive for at least one positive tuple $(\rho_1,\dots,\rho_s)$.
\end{lemma}
Therefore, we can verify that $\sum_\ell \rho_\ell \Lambda_\ell$ is weakly contractive by checking that $\sum_\ell   \Lambda_\ell$ is weakly contractive. 

This motivates the following definition:
\begin{definition}\label{def.weakContraction}
		Let a network $\Ns=(\Ss,\Rs)$ be given and assume that it admits a GLF $\tV(r)=\|B\Gamma r\|_\infty$. We say that the network is weakly contractive if there exist matrices $\Lambda_1,\dots,\Lambda_s$ satisfying Corollary \ref{cor} such that  $\bar\Lambda=\sum_{\ell=1}^s  \Lambda_\ell$ is weakly contractive.
	\end{definition}

\subsubsection{General results on strict contraction}
We state our result next which can be proved in a similar manner to the argument in Example 2 in \S \ref{s.strictcontraction}. All the test{s} stipulated in the statement of the theorem can be carried out computationally.
\begin{theorem}\label{th.scalednorm}
	Let a network $\Ns=(\Ss,\Rs)$ be given and assume that it admits a GLF $\tV(r)=\|B\Gamma r\|_\infty$, it has only trivial siphons, and is weakly contractive.
	
	Let $K\subset \Cs_\bx \cap \mathbb R_+^n$ be any compact set, and assume that $\varphi(t;K)$ is uniformly bounded (which follows automa{t}ically from the existence of a steady state or a conservation law as shown in Corollary \ref{cor.boundedness}). Then
	for any two solutions $x_1(\cdot)$ and $x_2(\cdot)$ of $\dot x=\Gamma R(x)$ that start in $K$, there exist $\bar\theta>0$, $c>0$ such that
	\[
	{\forall \theta \in (0,\bar\theta)}, |x_1(t)-x_2(t)|_{P^{{\theta}}B} \;\leq\; e^{-ct}|x_1(0)-x_2(0)|_{P^{{\theta}}B} \quad \mbox{for all}\ t\geq0,
	\]
	where $P^{{\theta}}$ is the corresponding contractor matrix defined before Lemma \ref{lem.contractor}.
\end{theorem}

\paragraph{Special case} Recall the three-body binding example. Remember that we did not need to scale the norm. In general, if a GLF takes the form $\tV(r)=\| \Theta \Gamma r \|_\infty$ for some diagonal matrix $\Theta$, then no scaling is required. This is stated by the following theorem which is proved in the appendix.
\begin{theorem}\label{th.strict_contraction}	Let a network $\Ns=(\Ss,\Rs)$ be given and assume that it admits a GLF $\tV(r)=\|\Theta\Gamma r\|_\infty$ for some non-negative diagonal matrix $\Theta$.  Furthermore, assume $\Ns$ admits trivial siphons only.

	Let $K\subset \Cs_\bx \cap \mathbb R_+^n$ be any compact set, and assume that $\varphi(t;K)$ is uniformly bounded. Then
	for any two solutions $x_1(\cdot)$ and $x_2(\cdot)$ of $\dot x=\Gamma R(x)$ that start in $K$, there exists   $c>0$ such that
	\[
	|x_1(t)-x_2(t)|_{\Theta} \;\leq\; e^{-ct}|x_1(0)-x_2(0)|_{\Theta} \quad \mbox{for all}\; t\geq0.
	\]
\end{theorem}

\section{Entrainment to periodic inputs}
In order to show entrainment for any arbitrary input, we need to show strict contraction. Nevertheless, this requirement has been relaxed using weaker notions of contraction \cite{margaliot16}. We review the basic results here, and then show that they apply to our class of systems. 

\subsection{Preliminaries}

\begin{definition}[\cite{margaliot16}]\label{sost} Consider a dynamical system $\dot x=f({x,t})$, where $f$ is defined on {a} convex forward invariant set $K\subset {\mathbb R}^n$, is differentiable in $x$, and is continuous in $t$. Let $|.|$ be a norm on $\mathbb R^n$. Then, $f$ is said to be contractive after a small overshoot and short transient (CSOST) on $K$ w.r.t. to $|.|$ if for each $\varepsilon>0$ and each $\tau>0$ there exists $c=c(\tau,\varepsilon)>{0}$ such that 
	\[ |\varphi(t+\tau;t_0,a)-\varphi(t+\tau;t_0,{b})| \le (1+\varepsilon) e^{-(t-t_0)c} |a-b|, \]
	for all $t\ge t_0$ and all $a,b \in K$. 
\end{definition} 
If a system is  CSOST then it entrains to periodic inputs as the following proposition states:

\begin{proposition}[\cite{margaliot16}]  \label{pr.entrainment}Let $T>0$. Given the system introduced in Definition \ref{sost}, and assume the domain $K$ is compact, and that $f$ is $T$-periodic. Then for any $t_0\ge0$ there exists a unique solution $\tilde x:[t_0,\infty)\to K$ with period $T$, and $\varphi(t;t_0,a)$ converges to $\tilde x$ for any $a\in K$. 
\end{proposition}

\subsection{Contraction after small transients and overshoots}

In order to study entrainment, a time-varying input is applied to the system \eqref{e.ode} via its parameters. A typical case is when the system is influenced by an upstream species, or by an external inflow. This can be modeled via time-varying kinetic constants. For example, $R_j(x,t)=k_j(t) \prod_{i=1}^n x_i^{\alpha_{ij}}$ in the case of mass-action kinetics.  But our formalism is not limited to these cases as it can be applied via any parameter appearing in a reaction rate. Therefore, we consider a more general form as follows:
\begin{equation}\label{e.tvode}
	\dot x= \Gamma R(x,t).
\end{equation}

To guarantee that our time-varying dynamical system behaves well, we assume the following:
\begin{enumerate}
	\item $R(x,t)$ is continuously differentiable in $x$, and continuous in $t$. 
	\item For each fixed $t^*\ge 0$, $R(x,t^*)$ satisfies AK1-AK3.
	
	\item For every reaction $\R_j$ there exist two rates $\bar R_j, \underline R_j$ satisfying AK1-AK3 such that the following is satisfied for all $t\ge 0$:
	\[ \underline R_j(x) \le R_j(x,t) \le \bar R_j(x).\]
\end{enumerate}

For uniformly bounded trajectories, persistence results generalize to the time-varying system \eqref{e.tvode} when the network admits no critical siphons \cite{angeli11}.  In our context, we need to establish a uniform bound{.}

\begin{lemma}\label{lem.sep}
	Let $\Ns$ be a given network, and assume that it admits trivial siphons only. Consider a fixed choice of admissible kinetics $R$, and consider the corresponding dynamical system \eqref{e.tvode}.   Fix a proper stoichiometric class $\Cs$ and let $K \subseteq \Cs$ be compact and forward invariant.  Then, for all $\tau>0$, $\varphi(\tau,K)$ is a compact subset of the open positive orthant. %
\end{lemma}
\begin{proof} 
	Fix $\tau >0$. Since $\varphi$ is continuous, then $\varphi(\tau;K)$ is compact. %
	To prove that it is a subset of the open positive orthant, assume by contradiction that $\exists x_1$ such that $x_1 \in \varphi(\tau;K) \cap \partial \mathbb R_{+}^n$. Let $x_0:=\varphi(-\tau;x_1)$. {Note that} we have $x_0\in \Cs$ {since the system is initialized in $\Cs$}. Note also that $x_1=\varphi(\tau;x_0) \in \partial \mathbb R_{+}^n$. Next, we claim that there is $\tau_1>0$ such that $\varphi(t;x_0) \in \mathbb R_+^n$ for all $t\in (0,\tau_1)$. The latter claim can be proven as follows: If no such $\tau_1$ exists, then there {is} an interval $J$ such that  $\varphi(t;x_0) \in \partial \mathbb R_+^n$ for all $t \in J$.   Hence, there is a face of the boundary that intersects the stoichiometry class  for which solutions evolve in such intersection for all $t \in J$. Hence, let us call this face $L_\Sigma=\{x: x_i=0, i\in \Sigma\}$ for some  $\Sigma\subset \Ss$. This means $\dot x_i=0$ for all $t \in J, i \in \Sigma$.%
	 Moreover outgoing reactions in which $i$ is a reactant are such that $R_j(x(t),t)=0$. Hence, all reactions that have $i$ as a product need to also have 0 rates (otherwise $\dot{x}_i(t) > 0$).
	This implies that all reactions feeding into $i$ have zero rates, hence they have a reactant in $\Sigma$, therefore $\Sigma$ is {a} non-trivial siphon which is a contradiction.
	
	Now, since $\varphi(\tau_1/2;x_0)\in \mathbb R_+^n$, then $\varphi(\tau - \tau_1/2; \varphi(\tau_1/2;x_0) )=x_1 \in \partial \mathbb R_+^n$. This implies that a trajectory starting in the interior reaches the boundary in finite time which contradicts the forward invariance of the open positive orthant. \hfill $\blacksquare$
	
\end{proof}
Next, we establish that the existence of a GLF guarantees contraction after small transients and overshoots:
\begin{theorem}[Contraction after small transients and overshoots]\label{th.glf_sost}
	Let a network $\Ns=(\Ss,\Rs)$ be given and assume that it admits a GLF $\tV(r)=\|B\Gamma r\|_\infty$, it has only trivial siphons, and is weakly contractive.  Let $\Cs$ be a stoichiometric class. Let $K\subset \Cs$ be compact and forward invariant. Then, the dynamical system \eqref{e.tvode} defined on $K$ is CSOST with respect to $|.|_B$ defined on $\Cs$.
\end{theorem}
\begin{proof}
	Fix $\tau,t_0>0$. %
	  Apply Lemma \ref{lem.sep}  to $K$, and  consider now the compact set $K_\tau=\varphi(\tau;t_0,K)$. Then, we apply Theorem \ref{th.scalednorm} for the trajectories starting at time $t_0+\tau$, i.e. starting in $K_\tau$. Hence, there exists some $\bar\theta(\tau)>0$ such that
	\begin{equation}\label{ent_proof1}
		|\varphi(t+\tau;t_0,a)-\varphi(t+\tau;t_0,b)|_{P_\theta B} \;\leq\; e^{-c(\tau,\theta) (t-t_0)}|\varphi(\tau;t_0,a)-\varphi(\tau;t_0,b)|_{P_\theta B} \quad \mbox{for all}\; t\geq0,
	\end{equation}
	where $P_\theta$ is  the diagonal contractor matrix defined before Lemma \ref{lem.contractor}.  Fix a choice of $\theta$. Then there exists $\tl\varepsilon=\tilde\varepsilon(\theta)>0$ (monotonically increasing as a function of  $\theta$) such that
	\[ |x|_B \le |x|_{P_\theta B} \le (1+\tilde\varepsilon(\theta) ) |x|_B.\]
	
	Applying the last inequality to \eqref{ent_proof1}, we get that the following holds for any choice of $\theta \in (0,\bar\theta)$:
	\begin{align}\nonumber |\varphi(t+\tau;t_0,a)-\varphi(t+\tau;t_0,b)|_{B} &\le |\varphi(t+\tau;t_0,a)-\varphi(t+\tau;t_0,b)|_{P_\theta B}  \\ \nonumber & \le  e^{-c(\tau,\theta) (t-t_0)}|\varphi(\tau;t_0,a)-\varphi(\tau;t_0,b)|_{P_\theta B} \\ \nonumber & \le (1+\tilde\varepsilon(\theta) ) e^{-c(\tau,\theta) (t-t_0)}|\varphi(\tau;t_0,a)-\varphi(\tau;t_0,b)|_{B} \\ \label{ent_proof2} & \mathop{\le}\limits^{(\star)}(1+\tilde\varepsilon(\theta) ) e^{-c(\tau,\theta) (t-t_0)}|a-b|_{B}, \end{align}
	where $(\star)$ follows from nonexpansiveness. 
	
	Pick an arbitrary $\varepsilon>0$. If $\varepsilon<\tilde\varepsilon{({\bar\theta})}$, then there exists $\theta \in (0,\bar\theta)$ such that $\varepsilon=\tilde\varepsilon(\theta)$. Hence, using \eqref{ent_proof2} we get:
	\begin{equation*} |\varphi(t+\tau;t_0,a)-\varphi(t+\tau;t_0,b)|_{B}  \le (1+\varepsilon ) e^{-c(\tau,\varepsilon)(t-t_0)}|a-b|_{B}.  \end{equation*}
	If $\varepsilon\ge\tilde\varepsilon(\bar\theta)$, then we apply \eqref{ent_proof2} for $\tl\varepsilon= \tl\varepsilon({\bar\theta/2})<\varepsilon$. Hence, we get
	\begin{equation*} |\varphi(t+\tau;t_0,a)-\varphi(t+\tau;t_0,b)|_{B}  \le (1+\tilde\varepsilon({\bar\theta/2}) )   e^{-c(\tau,\bar\theta(\tau)/2) (t-t_0)}|a-b|_{B} \le  (1+\varepsilon ) e^{-c(\tau)(t-t_0)}|a-b|_{B}.  \end{equation*}
	Therefore, the dynamical system is CSOST on $K$.
\end{proof}

\subsection{Entrainment inside a stoichiometric class}
Finally, we can state our main result which follows by combining Theorem \ref{th.glf_sost} with Proposition \ref{pr.entrainment}:
\begin{theorem}\label{th.entrainment}
	Let a network $\Ns=(\Ss,\Rs)$ be given and assume that it admits a GLF $\tV(r)=\|B\Gamma r\|_\infty$, it has only trivial siphons, and is weakly contractive.
	Let $\Cs$ be a stoichiometric class. Let $K\subset \Cs$ be compact and forward invariant.
	Let $T>0$. Assume that $R(x,t)$   is $T$-periodic. Then for any $t_0\ge0$ there exists a unique solution $\tilde x:[t_0,\infty)\to K$ with period $T$, and $\varphi(t;t_0,a)$ converges to $\tilde x$ for any $a\in K$. 
\end{theorem}

\begin{remark}
	In \cite{margaliot16}, the notion of contraction after small transients has been introduced. In our context, if a network $\Ns$ admits a GLF of the form $\tV(r)=\|\Theta \Gamma r\|_\infty$ for some nonzero nonnegative diagonal $\Theta$,  then it is contractive with small transients. This is since no scaling of the norm is required as shown in Theorem \ref{th.strict_contraction}.
\end{remark}

\section{Biochemical Examples}

Although our methods are very general and can apply to networks in the micro- and macro-scales, we   focus on several well-established and accepted models of core biochemical modules which are involved in   cellular processes. They constitute     elemental building blocks of biochemical pathways and  cascades. %

 We review next some biochemical motifs which are amenable to our results. {The examples are depicted graphically in Figure \ref{f.networks} using the Petri-net convention \cite{petri08}, where a reaction is represented by a {rectangle}, and a species is repres{e}nted by a circle.} Other examples and more complex cascades can be analyzed too as discussed in \cite{MA_LEARN,MA_MCSS23}. All the examples presented in this section are conservative and have only trivial siphons.
\begin{figure}[t]
	\centering
	\subfigure[A single PTM cycle]{\scalebox{0.5515}{%
\begin{tikzpicture}[node distance=1.5cm,>=stealth',bend angle=45,auto]
\tikzstyle{S}=[circle,thick,draw=black,fill=black,minimum size=2.5mm]
\tikzstyle{S1}=[circle,thick,draw=gray,fill=gray,minimum size=2.5mm]
\tikzstyle{R}=[rectangle,very thin,draw=black,
fill=black,minimum width=.01in, minimum height=9mm]
\begin{scope}
\node [S1] (XE_1)    {};

\node [R] (R1) [right of=XE_1,label=above:$R_2$] {}
edge [<-]                  (XE_1);
\node [R] (R2) [left of=XE_1,label=above:$R_1$] {}
edge [<->]                  (XE_1);

\node [S1] (YE_2) [below of=XE_1] {};
\node [R] (R3) [right of=YE_2,label=above:$R_3$] {}
edge [<->]                  (YE_2);
\node [R] (R4) [left of=YE_2,label=above:$R_4$] {}
edge [<-]                  (YE_2);

\node [S] (Y) [below=0.8cm, right=2cm, label=right:$S^+$]   {};
\node [S] (X4) [below=0.8cm, left=2cm, label=left:$S$]   {};

\node [S] (E_1) [above of=XE_1, label=below:Kinase]   {};
\node [S] (E_2) [below of=YE_2, label=above:Phosphatase]   {};

\draw[->] (R1) .. controls +(0:0.4cm) and +(90:0.4cm) .. (Y);
\draw[<->] (Y) .. controls +(-90:0.4cm) and +(20:0.4cm) .. (R3);
\draw[->] (R4) .. controls +(-180:0.4cm) and +(-90:0.4cm) .. (X4);
\draw[<->] (X4) .. controls +(+90:0.4cm) and +(-160:0.4cm) .. (R2);
\draw[->] (R1) .. controls +(20:1.3cm) and +(0:1.3cm) .. (E_1);
\draw[<->] (E_1) .. controls +(-180:1.3cm) and +(160:1.3cm) .. (R2);
\draw[->] (R4) .. controls +(200:1.3cm) and +(180:1.3cm) .. (E_2);
\draw[<->] (E_2) .. controls +(0:1.3cm) and +(-20:1.3cm) .. (R3);
\end{scope}
\end{tikzpicture}  %
}}
	\subfigure[Three-Body Binding]{	 {\includegraphics[width=0.182\textwidth]{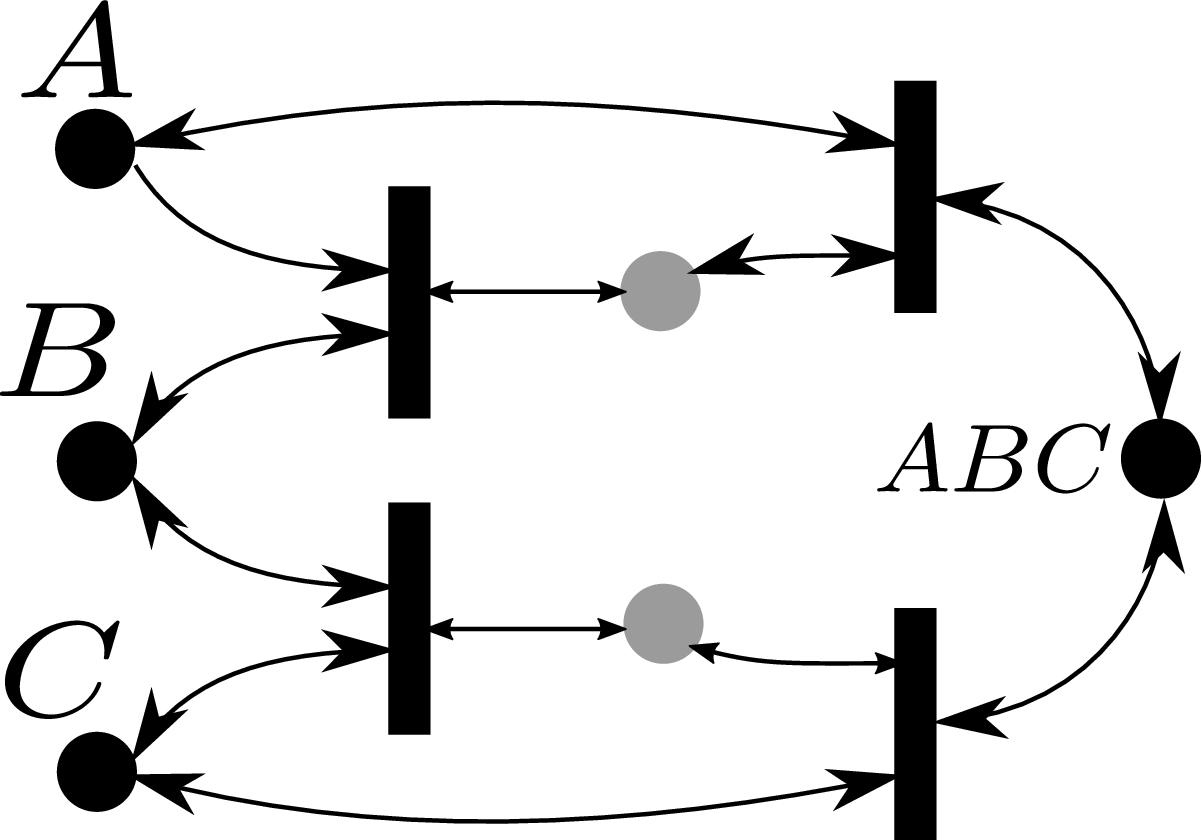} }} 
	\subfigure[Proofreading network]{{\includegraphics[width=0.262\textwidth]{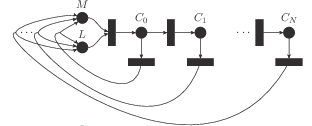} } }
	\subfigure[A Phosphorelay]{\includegraphics[width=0.28\textwidth]{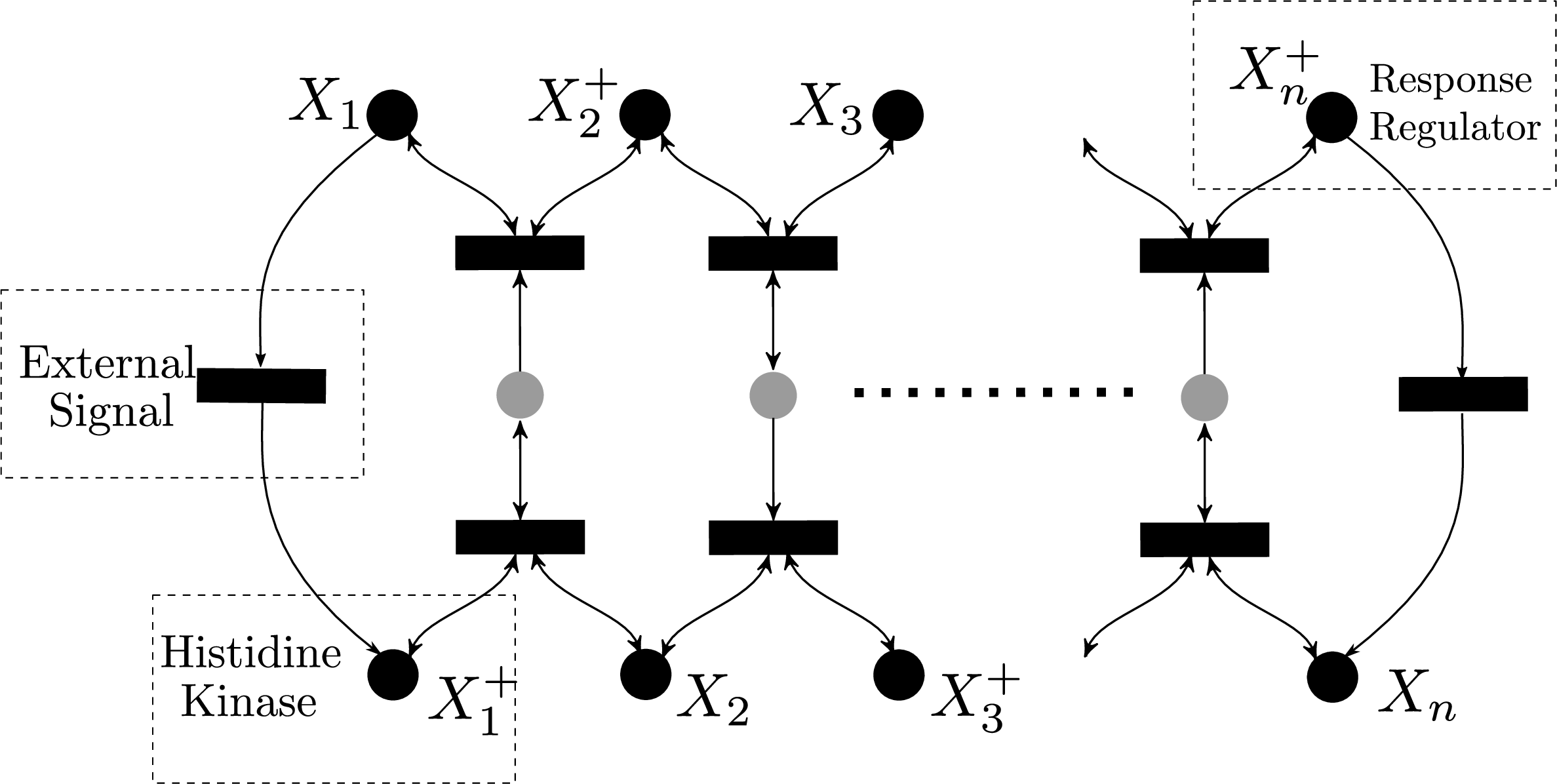}}
	\caption{Examples of nonlinear biochemical motifs analyzed in section 6.}\label{f.networks}
\end{figure}

\subsection{Post-translational modification (PTM)}  After a protein is translated from mRNA, it can be enzymatically modified by the attachment of a chemical group (e.g., a phosphate) to a specific amino-acid residue. PTMs are involved in regulating many cellular processes, since they allow a protein to change its behavior to respond in a timely manner to environmental changes and cellular signals. Fig.~\ref{f.networks}(a) depicts a PTM of a protein substrate $S$ based on the widely-accepted model introduced in 1981 \cite{goldbeter81}, and has been experimentally studied  \cite{cimino87} soon after. It has been widely used to describe various biochemical systems  \cite{meinke86,ferrell96,kholodenko02,iglesias02,samoilov05,jiang11}.  The long-term dynamics of the single PTM cycle have been first studied in \cite{angeli08} using monotone system techniques. Multiple cascades and alternative forms have been studied in \cite{MA_LEARN,MA_MCSS23}.

A simplified network was already presented in \eqref{e.ptm}. Here, we consider the full PTM given below:
	\begin{equation}\label{e.ptm_rev} \begin{array}{rl}  S+ E \xrightleftharpoons[]{} C_1 \xrightarrow{} P+E \\ P+D \xrightleftharpoons[]{} C_2 \xrightarrow{} S+D \end{array}.\end{equation}
	
	As mentioned in \S 2.3, we can write a max-min GLF for network in the form $\tV(r)=\|Cr\|_{\infty}$. The stoichiometric matrix $\Gamma$ and the matrix $C$ can be written as follows:
	{\tiny \[ \Gamma=    \left[ \begin{array}{rrrrrr} -1 & 1 & 0 & 0 & 0 & 1\\ -1 & 1 & 1 & 0 & 0 & 0\\ 1 & -1 & -1 & 0 & 0 & 0\\ 0 & 0 & 1 & -1 & 1 & 0\\ 0 & 0 & 0 & -1 & 1 & 1\\ 0 & 0 & 0 & 1 & -1 & -1 \end{array}\right], ~ C=    \left[\begin{array}{rrrrrr} 0 & 0 & 1 & 0 & 0 & -1\\ -1 & 1 & 1 & 0 & 0 & 0\\ 0 & 0 & 1 & -1 & 1 & 0\\ -1 & 1 & 0 & 0 & 0 & 1\\ 0 & 0 & 0 & -1 & 1 & 1\\ 1 & -1 & 0 & -1 & 1 & 0 \end{array}\right].
	\]}
Using Corollary \ref{cor}, we find matrices $\Lambda_1,\dots,\Lambda_8$ via the subroutine \texttt{FindLambda.m} {using the command \texttt{FindLambda($\Gamma$,C)}}:
{\tiny 
\begin{align*}
\Lambda_1&=\left[\begin{array}{cccccc} 0 & 0 & 0 & 0 & 0 & 0\\ 1 & -1 & 0 & 0 & 0 & 0\\ 0 & 0 & 0 & 0 & 0 & 0\\ 0 & 0 & 0 & -1 & 0 & 0\\ 0 & 0 & 0 & 0 & 0 & 0\\ 0 & 0 & 0 & 0 & 1 & -1 \end{array}\right],
\Lambda_2=\left[\begin{array}{cccccc} 0 & 0 & 0 & 0 & 0 & 0\\ 0 & -1 & 0 & 0 & 0 & 0\\ 0 & 0 & 0 & 0 & 0 & 0\\ -1 & 0 & 0 & -1 & 0 & 0\\ 0 & 0 & 0 & 0 & 0 & 0\\ 0 & 0 & 1 & 0 & 0 & -1 \end{array}\right],
\Lambda_3=\left[\begin{array}{cccccc} 0 & 0 & 0 & 0 & 0 & 0\\ 0 & -1 & 0 & 0 & 0 & 0\\ 0 & 0 & 0 & 0 & 0 & 0\\ -1 & 0 & 0 & -1 & 0 & 0\\ 0 & 0 & 0 & 0 & 0 & 0\\ 0 & 0 & 1 & 0 & 0 & -1 \end{array}\right], \\
\Lambda_4&=\left[\begin{array}{cccccc} -1 & 0 & 0 & -1 & 0 & 0\\ 0 & -1 & 0 & 0 & 0 & 0\\ 0 & 0 & -1 & 0 & 0 & 1\\ 0 & 0 & 0 & 0 & 0 & 0\\ 0 & 0 & 0 & 0 & 0 & 0\\ 0 & 0 & 0 & 0 & 0 & 0 \end{array}\right],
\Lambda_5=\left[\begin{array}{cccccc} 0 & 0 & 0 & 0 & 0 & 0\\ 0 & 0 & 0 & 0 & 0 & 0\\ 0 & 0 & -1 & 0 & 0 & 0\\ 0 & 0 & 0 & 0 & 0 & 0\\ -1 & 0 & 0 & 0 & -1 & 0\\ 0 & -1 & 0 & 0 & 0 & -1 \end{array}\right],
\Lambda_6=\left[\begin{array}{cccccc} 0 & 0 & 0 & 0 & 0 & 0\\ 0 & 0 & 0 & 0 & 0 & 0\\ 1 & 0 & -1 & 0 & 0 & 0\\ 0 & 0 & 0 & 0 & 0 & 0\\ 0 & 0 & 0 & 0 & -1 & 0\\ 0 & 0 & 0 & -1 & 0 & -1 \end{array}\right], \\
\Lambda_7&=\left[\begin{array}{cccccc} 0 & 0 & 0 & 0 & 0 & 0\\ 0 & 0 & 0 & 0 & 0 & 0\\ 1 & 0 & -1 & 0 & 0 & 0\\ 0 & 0 & 0 & 0 & 0 & 0\\ 0 & 0 & 0 & 0 & -1 & 0\\ 0 & 0 & 0 & -1 & 0 & -1 \end{array}\right],
\Lambda_8=\left[\begin{array}{cccccc} -1 & 0 & 1 & 0 & 0 & 0\\ 0 & 0 & 0 & 0 & 0 & 0\\ 0 & 0 & 0 & 0 & 0 & 0\\ 0 & 0 & 0 & -1 & 0 & -1\\ 0 & 0 & 0 & 0 & -1 & 0\\ 0 & 0 & 0 & 0 & 0 & 0 \end{array}\right]
\end{align*}
}
Using Lemma \ref{mainlemmaB}, the logarithmic norm of the Jacobian is upper bounded by 
$\mu_\infty(\sum_{\ell=1}^s \rho_\ell \Lambda_\ell)$, where $\sum_{\ell=1}^s \rho_\ell \Lambda_\ell$ can be written as:
{\tiny \[    \left[\begin{array}{cccccc} -\rho_{4}-\rho_{8} & 0 & \rho_{8} & -\rho_{4} & 0 & 0\\ \rho_{1} & -\rho_{1}-\rho_{2}-\rho_{3}-\rho_{4} & 0 & 0 & 0 & 0\\ \rho_{6}+\rho_{7} & 0 & -\rho_{4}-\rho_{5}-\rho_{6}-\rho_{7} & 0 & 0 & \rho_{4}\\ -\rho_{2}-\rho_{3} & 0 & 0 & -\rho_{1}-\rho_{2}-\rho_{3}-\rho_{8} & 0 & -\rho_{8}\\ -\rho_{5} & 0 & 0 & 0 & -\rho_{5}-\rho_{6}-\rho_{7}-\rho_{8} & 0\\ 0 & -\rho_{5} & \rho_{2}+\rho_{3} & -\rho_{6}-\rho_{7} & \rho_{1} & -\rho_{1}-\rho_{2}-\rho_{3}-\rho_{5}-\rho_{6}-\rho_{7} \end{array}\right]
\]}
Note that $\mu_\infty(\sum_{\ell=1}^s \rho_\ell \Lambda_\ell)=0$. Using the notation defined in \S \ref{s.weakcontraction}, we {analyze this network manually to find that} it is weakly contractive with $S_0=\{1,6\}$. Therefore the contractor matrix is given as $P_\theta=\diag([1,1+\theta,1+\theta,1+\theta,1+\theta,1])$. The scaled norm is given as $P_\theta B$. Remember that $B$ is not {unique}, one such choice provided by the subroutine \texttt{FindBfromC.m} in the package \texttt{LEARN} is:
    {\tiny\[ P_\theta B= \frac 14 \left[\begin{array}{cccccc} -2 & 1 & -1 & 2 & -1 & 1\\ 0 & 2\,\theta +2 & -2 \theta -2 & 0 & 0 & 0\\ 0 & 0 & 0 & 4 \theta +4 & 0 & 0\\ 4 \theta +4 & 0 & 0 & 0 & 0 & 0\\ 0 & 0 & 0 & 0 & 2 \theta +2 & -2 \theta -2\\ -2 & -1 & 1 & 2 & 1 & -1 \end{array}\right].\]}
    
Hence, for an {arbitrary} given positive compact set in a stoichiometric class, the trajectories are strictly contractive with respect to the $P_\theta B$-weighted $\infty$-norm, where $\theta>0$ is small enough depending on the choice of the compact set. In addition, the trajectories  entrain to periodic time-varying parameters. 

\subsection{Three-Body Binding} 

As introduced before,  \emph{three-body binding} \cite{douglass13} describes the binding of two molecules via a bridging molecule to form a ternary complex. Examples include T-cell receptors with toxins \cite{saline10}  and supramolecular assembly \cite{sagar17}. 

Consider the three-body network shown in \eqref{three_body} and depicted in Figure \ref{f.networks}(b). 
The reactions are listed in \eqref{three_body}, let $x_1=[{\rm A}], x_2=[{\rm B}], x_3=[{\rm C}], x_4=[{\rm AB}], x_5=[{\rm BC}],x_6=[{\rm ABC}].$%
  Using \texttt{FindLambda} subroutine, the matrices $\Lambda_1,\dots,\Lambda_{12}$ are given {using the command \texttt{FindLambda($\Gamma$,C)}}: 
{\tiny \begin{align*}\Lambda_1&=
\left[\begin{array}{cccccc} -1 & 0 & 0 & 0 & 0 & 0\\ 0 & -1 & 0 & 0 & -1 & 0\\ 0 & 0 & 0 & 0 & 0 & 0\\ 0 & 0 & 0 & -1 & 0 & -1\\ 0 & 0 & 0 & 0 & 0 & 0\\ 0 & 0 & 0 & 0 & 0 & 0 \end{array}\right],
\Lambda_2=\left[\begin{array}{cccccc} -1 & 0 & 0 & 0 & 0 & 0\\ 0 & 0 & 0 & 0 & 0 & 0\\ 0 & 0 & 0 & 0 & 0 & 0\\ 0 & 0 & 0 & 0 & 0 & 0\\ 0 & -1 & 0 & 0 & -1 & 0\\ 0 & 0 & 0 & -1 & 0 & -1 \end{array}\right], 
\Lambda_3=
\left[\begin{array}{cccccc} -1 & 0 & 0 & 0 & 1 & 0\\ 0 & -1 & 0 & 0 & 0 & 0\\ 0 & 0 & 0 & 0 & 0 & 0\\ 0 & 0 & 1 & -1 & 0 & 0\\ 0 & 0 & 0 & 0 & 0 & 0\\ 0 & 0 & 0 & 0 & 0 & 0 \end{array}\right], \\
\Lambda_4&=\left[\begin{array}{cccccc} 0 & 0 & 0 & 0 & 0 & 0\\ 0 & -1 & 0 & 0 & 0 & 0\\ 0 & 0 & -1 & 1 & 0 & 0\\ 0 & 0 & 0 & 0 & 0 & 0\\ 1 & 0 & 0 & 0 & -1 & 0\\ 0 & 0 & 0 & 0 & 0 & 0 \end{array}\right],  
\Lambda_5  =
\left[\begin{array}{cccccc} 0 & 0 & 0 & 0 & 0 & 0\\ 0 & -1 & 0 & -1 & 0 & 0\\ 0 & 0 & -1 & 0 & 0 & 0\\ 0 & 0 & 0 & 0 & 0 & 0\\ 0 & 0 & 0 & 0 & -1 & -1\\ 0 & 0 & 0 & 0 & 0 & 0 \end{array}\right], 
\Lambda_6=
\left[\begin{array}{cccccc} 0 & 0 & 0 & 0 & 0 & 0\\ 0 & 0 & 0 & 0 & 0 & 0\\ 0 & 0 & -1 & 0 & 0 & 0\\ 0 & -1 & 0 & -1 & 0 & 0\\ 0 & 0 & 0 & 0 & 0 & 0\\ 0 & 0 & 0 & 0 & -1 & -1 \end{array}\right], \\
\Lambda_7&=
\left[\begin{array}{cccccc} -1 & 0 & 0 & 0 & 0 & -1\\ 0 & -1 & 1 & 0 & 0 & 0\\ 0 & 0 & 0 & 0 & 0 & 0\\ 0 & 0 & 0 & -1 & 0 & 0\\ 0 & 0 & 0 & 0 & 0 & 0\\ 0 & 0 & 0 & 0 & 0 & 0 \end{array}\right],
\Lambda_8=
\left[\begin{array}{cccccc} 0 & 0 & 0 & 0 & 0 & 0\\ 0 & 0 & 0 & 0 & 0 & 0\\ 0 & 1 & -1 & 0 & 0 & 0\\ 0 & 0 & 0 & -1 & 0 & 0\\ 0 & 0 & 0 & 0 & 0 & 0\\ -1 & 0 & 0 & 0 & 0 & -1 \end{array}\right], 
\Lambda_9 =
\left[\begin{array}{cccccc} -1 & 1 & 0 & 0 & 0 & 0\\ 0 & 0 & 0 & 0 & 0 & 0\\ 0 & 0 & 0 & 0 & 0 & 0\\ 0 & 0 & 0 & 0 & 0 & 0\\ 0 & 0 & 0 & 0 & -1 & 0\\ 0 & 0 & -1 & 0 & 0 & -1 \end{array}\right], \\
\Lambda_{10}&=
\left[\begin{array}{cccccc} 0 & 0 & 0 & 0 & 0 & 0\\ 1 & -1 & 0 & 0 & 0 & 0\\ 0 & 0 & -1 & 0 & 0 & -1\\ 0 & 0 & 0 & 0 & 0 & 0\\ 0 & 0 & 0 & 0 & -1 & 0\\ 0 & 0 & 0 & 0 & 0 & 0 \end{array}\right], 
\Lambda_{11}
\left[\begin{array}{cccccc} 0 & 0 & 0 & 0 & 0 & 0\\ 0 & 0 & 0 & 0 & 0 & 0\\ 0 & 0 & -1 & 0 & -1 & 0\\ -1 & 0 & 0 & -1 & 0 & 0\\ 0 & 0 & 0 & 0 & 0 & 0\\ 0 & 0 & 0 & 0 & 0 & -1 \end{array}\right], 
\Lambda_{12} =
\left[\begin{array}{rrrrrr} -1 & 0 & 0 & -1 & 0 & 0\\ 0 & 0 & 0 & 0 & 0 & 0\\ 0 & 0 & 0 & 0 & 0 & 0\\ 0 & 0 & 0 & 0 & 0 & 0\\ 0 & 0 & -1 & 0 & -1 & 0\\ 0 & 0 & 0 & 0 & 0 & -1 \end{array}\right].
\end{align*} }
Therefore the expression $\sum_{\ell=1}^s \rho_\ell \Lambda_\ell$ can be written {using the command \texttt{FindLambda($\Gamma$,C)}}:
{\tiny\[   \left[\begin{array}{rrrrrr} -{\displaystyle\sum_{k\in\{1,2,3,7,9,12\}}} \!\! \rho_{k} & \rho_{9} & 0 & -\rho_{12} & \rho_{3} & -\rho_{7}\\ \rho_{10} & 
-{\displaystyle\sum_{k\in\{1,3,4,5,7,10\}}}\!\!  \rho_k & \rho_{7} & -\rho_{5} & -\rho_{1} & 0\\ 0 & \rho_{8} & 
-{\displaystyle\sum_{k\in\{4,5,6,8,10,11\}}} \!\! \rho_k  & \rho_{4} & -\rho_{11} & -\rho_{10}\\ -\rho_{11} & -\rho_{6} & \rho_{3} & 
-{\displaystyle\sum_{k\in\{1,3,6,7,8,11\}}} \!\! \rho_{k}   & 0 & -\rho_{1}\\ \rho_{4} & -\rho_{2} & -\rho_{12} & 0 & 
-{\displaystyle\sum_{k\in\{2,4,5,9,10,12\}}} \!\! \rho_{k}   & -\rho_{5}\\ -\rho_{8} & 0 & -\rho_{9} & -\rho_{2} & -\rho_{6} & 
-{\displaystyle\sum_{k\in\{2,6,8,9,11,12\}}} \!\! \rho_{k}   \end{array}\right] 
\]}

Then, we get {(provided by \texttt{FindLambda.m})}:
\[\mu_{\infty}\left (\sum_\ell \rho_\ell \Lambda_\ell \right)\le -\min\{\rho_{1}+\rho_{2},\rho_{3}+\rho_{4},\rho_{5}+\rho_{6}, \rho_{7}+\rho_{8},\rho_{9}+\rho_{10},\rho_{11}+\rho_{12}\} <0, \]
which is strictly negative over any positive compact set. Note $P_\theta=I, B=I$. Hence, for an {arbitrary} given positive compact set in a stoichiometric class, the trajectories are strictly contractive with respect to the standard $\infty$-norm. In addition, the trajectories entrain to periodic time-varying parameters. 

{As a numerical illustration of the results, we depict simulations of the three-body binding system in Figure \ref{f.three_Body_sim}. Figure \ref{f.three_Body_sim}(a) shows the non-expansivity of the norm with $B=I$. With 1000 numerically generated trajectories, the maximum logarithmic norm recorded is   $-7.1636\times 10^{-5} $.  Figure \ref{f.three_Body_sim}(b) shows the trajectories converging to the same steady state in the stoichiometric class.  Figure \ref{f.three_Body_sim}(c) shows that all the states entrain to a periodic change in the parameters.
}

\begin{figure}
	 \subfigure[]{\includegraphics[width=0.32\textwidth]{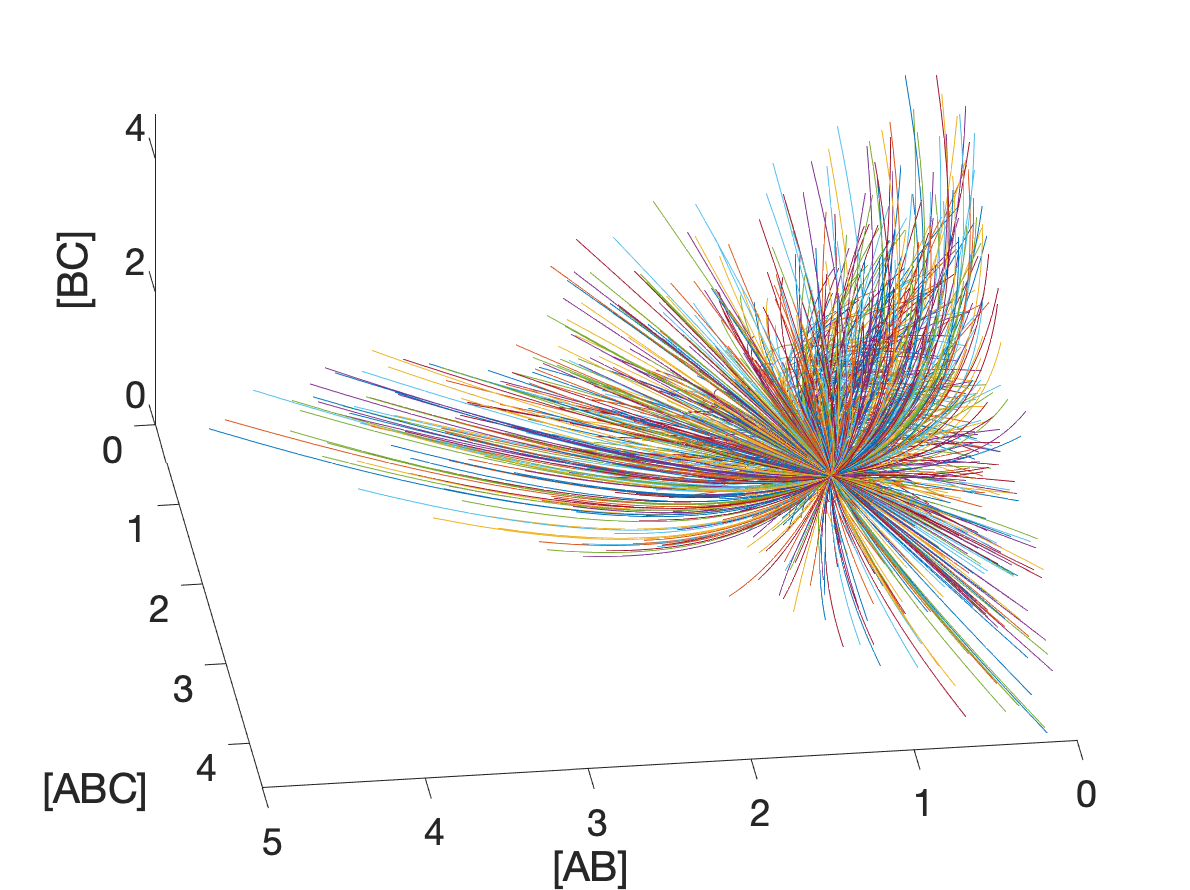}}
	 	 \subfigure[]{\includegraphics[width=0.32\textwidth]{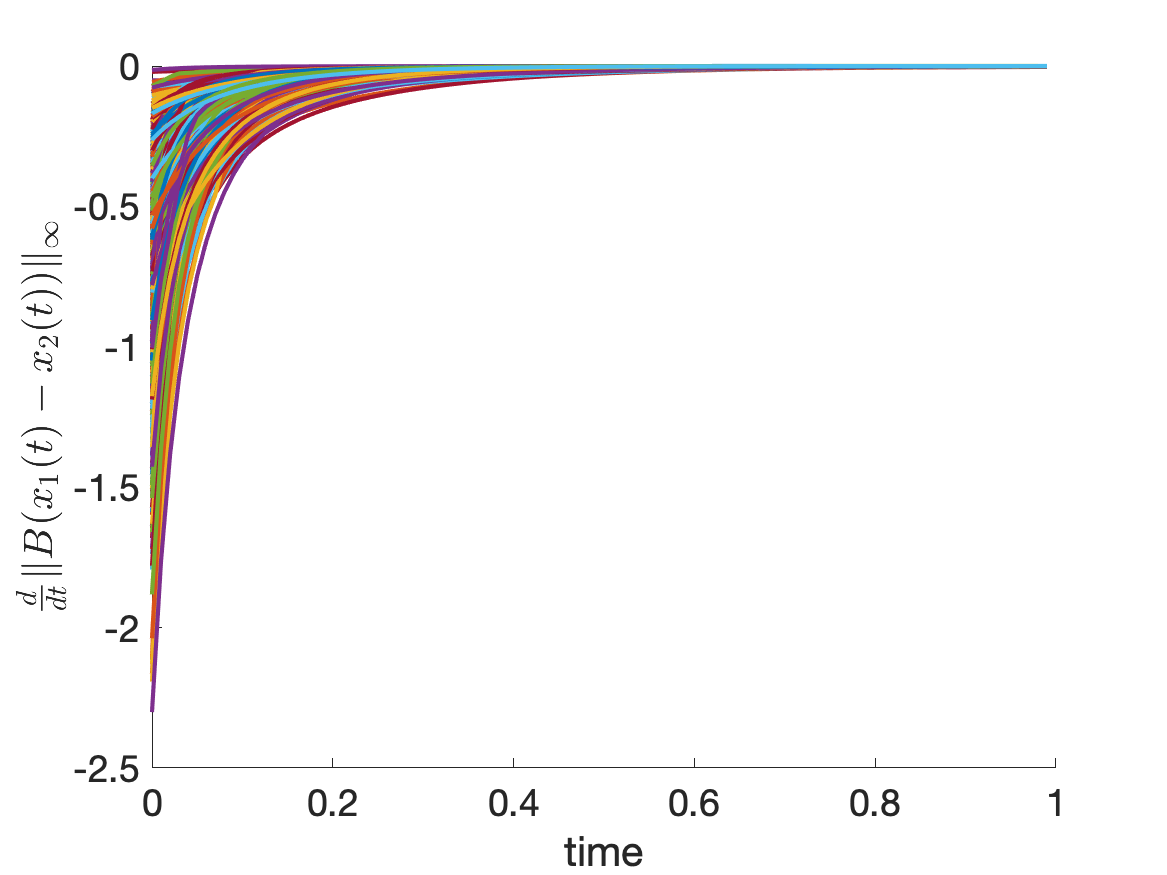}}
	 	 	 \subfigure[]{\includegraphics[width=0.32\textwidth]{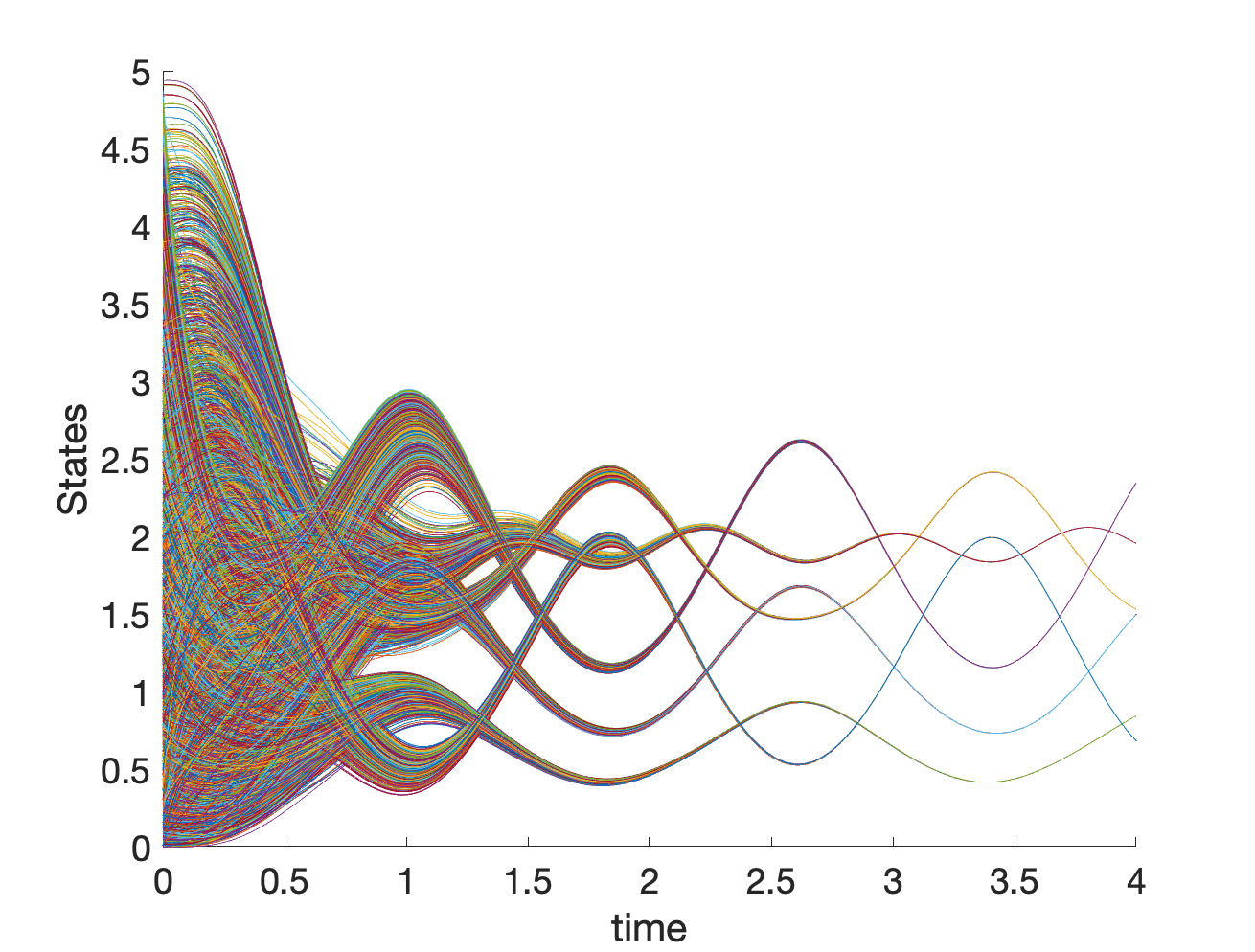}}
	 	 	 \caption{{Simulations of the three-body binding network using the norm found by \texttt{LEARN}. (a) A simulation of 1000 trajectories in the stoichiometric class {corresponding} to $[A]_{tot}=[B]_{tot}=[C]_{tot}=5$, for the ODE with mass-action kinetics and all kinetic constants set to 1. (b) The numerically-calculated time-derivative of the $B$-distance between 1000 pairs of {randomly} generated trajectories for the same ODE system in panel (a). Each pair of trajectories start{s} in the same stoichiometric class. (c) Time-series plots of all six states with $[A]_{tot}=[B]_{tot}=[C]_{tot}=5$ and reaction-rates following mass-action kinetics. The periodically-varying kinetic constants are given as $k_1=k_5=k_6=k_3=k_8=\sin^2(2t), k_2=k_4=k_7=\cos^2(2t)$.}}
	 	 	 \label{f.three_Body_sim}
\end{figure}
 
\subsection{Kinetic proo{f}reading in $T$-cell activation}   The kinetic proofreading model introduced {was} by \cite{hopfield74}. It was repurposed by McKeithan \cite{mckeithan95,sontag01} to model $T$-cell activation, and is supported experimentally \cite{gaud18}. In this model, a ligand $M$ binds to a $T$-cell receptor. The receptor-ligand complex then undergoes a series of reactions to produce an active complex $C_N$.  Figure~\ref{f.networks}(c) depicts the {BIN}. The reactions can be written as follows:
\begin{align*} M+L & \rightleftharpoons C_0 \lra C_1 \lra \dots \lra C_N \\
C_1 & \lra M+L, C_2 \lra M+L,\dots, C_N \lra M+L .
\end{align*}

Let $N=2$. We use \texttt{LEARN} to find the matrix $C$. Hence,  the pair $\Gamma,C$ can be written as follows:
{\tiny \[      \Gamma=\left[\begin{array}{rrrrrr} -1 & 1 & 0 & 0 & 1 & 1\\ -1 & 1 & 0 & 0 & 1 & 1\\ 1 & -1 & -1 & 0 & 0 & 0\\ 0 & 0 & 1 & -1 & -1 & 0\\ 0 & 0 & 0 & 1 & 0 & -1 \end{array}\right],
	 ~     C= \left[\begin{array}{rrrrrr} 0 & 0 & 0 & 1 & 0 & -1\\ 0 & 0 & 1 & -1 & -1 & 0\\ 0 & 0 & 1 & 0 & -1 & -1\\ 1 & -1 & -1 & 0 & 0 & 0\\ 1 & -1 & -1 & 1 & 0 & -1\\ 1 & -1 & 0 & -1 & -1 & 0\\ 1 & -1 & 0 & 0 & -1 & -1 \end{array}\right].
	 \]}
	 
The matrices $\Lambda_1, \dots, \Lambda_7$ can be found, and the expression $\sum_\ell \rho_\ell \Lambda_\ell$ can be written {using the command \texttt{FindLambda($\Gamma$,C)}}:
{\tiny
\[
\left[\begin{array}{rrrrrrc} -\rho_{5}-\rho_{7} & 0 & \rho_{5} & 0 & 0 & 0 & 0\\ 0 & -\rho_{4}-\rho_{5}-\rho_{6} & 0 & 0 & 0 & \rho_{4} & 0\\ \rho_{6} & \rho_{7} & -\rho_{4}-\rho_{6}-\rho_{7} & 0 & 0 & 0 & \rho_{4}\\ 0 & 0 & -\rho_{1}-\rho_{2} & -\rho_{1}-\rho_{2}-\rho_{3}-\rho_{4} & 0 & 0 & 0\\ \rho_{3}+\rho_{4} & -\rho_{1}-\rho_{2} & 0 & \rho_{7} & -\rho_{1}-\rho_{2}-\rho_{3}-\rho_{4}-\rho_{5}-\rho_{7} & 0 & \rho_{5}\\ -\rho_{1}-\rho_{2} & \rho_{3} & 0 & \rho_{5}+\rho_{6} & 0 & -\rho_{1}-\rho_{2}-\rho_{3}-\rho_{5}-\rho_{6} & 0\\ 0 & 0 & \rho_{3} & 0 & \rho_{6} & \rho_{7} & -\rho_{1}-\rho_{2}-\rho_{3}-\rho_{6}-\rho_{7} \end{array}\right].
\]
}
Note that $\mu_\infty(\sum_\ell \rho_\ell \Lambda_\ell)=0$. {We {perform} weak contractivity analysis manually.} The network is weakly contractive with $S_0=\{3,5,6\}$ and $S_-=\{1,2,4,7\}$.  We find a matrix $B$ satisfying $C=B\Gamma$ using {$\texttt{FindBfromC.m}$}, and the weighting matrix for the norm can be written as:
\[
P_\theta B =     \frac 14 \left[\begin{array}{ccccc} 0 & 0 & 0 & 0 & 4\,\theta +4\\ 0 & 0 & 0 & 4\,\theta +4 & 0\\ -1 & -1 & -2 & 2 & 2\\ 0 & 0 & 4\,\theta +4 & 0 & 0\\ -1 & -1 & 2 & -2 & 2\\ -1 & -1 & 2 & 2 & -2\\ -2\,\theta -2 & -2\,\theta -2 & 0 & 0 & 0 \end{array}\right]. 
\]
Hence, for an {arbitrary} given positive compact set in a stoichiometric class, the trajectories are strictly contractive with respect to the $P_\theta B$-weighted $\infty$-norm, where $\theta>0$ is small enough depending on the choice of the compact set. More exactly, let $\wt K$ be the minimal compact set of the form \eqref{e.K_tilde} that contains $K$. Then, the upper bound $\bar \theta$ that appears in Theorem \ref{th.strict_contraction} can be written as:
{\[ \bar\theta=\min_{x\in\wt K}\left\{ \frac{\rho_7}{\rho_5}, \frac{\rho_5+\rho_6}{\rho_4},\frac{\rho_4+\rho_3}{\rho_1+\rho_2},\frac{\rho_2+\rho_1}{\rho_3+\rho_6+\rho_7}\right\}>0, \]}
and the contractivity constant appearing in Theorem \ref{th.strict_contraction} can be written as:
{\footnotesize\begin{align*} c&= -\min_{x\in\wt K}\left\{ \rho_7 - \theta \rho_5, \rho_6+\rho_5 - \theta \rho_5, \frac{\theta}{\theta+1}(\rho_7+\rho_6+\rho_4),\rho_4+\rho_3-\theta(\rho_1+\rho_2), \frac{\theta}{\theta+1} (\rho_7+\rho_5+\rho_4+\rho_3+\rho_2+\rho_1), \right .\\ & \qquad \qquad \left . \frac{\theta}{\theta+1} (\rho_6+\rho_5+\rho_3+\rho_2+\rho_1), \rho_1+\rho_2 - \theta(\rho_3+\rho_6+\rho_7)\right \}<0, \end{align*}}
where the inequality holds for all $\theta \in(0,\bar\theta)$. 
 In addition, the trajectories  entrain to periodic time-varying parameters.

\subsection{{Phosphorelays}}  	Unlike advanced organisms which rely on PTMs for signalling, bacteria often use a motif known as \textit{phosphotransfer} where   a histidine kinase provides its phosphate   to a response regulator. This motif is a building block in what is known as a two-component signaling system \cite{stock00}. Experimental examples include EnvZ/OmpR system in E. Coli \cite{batchelor03}.  Phosphotransfer cascades are known as phosphorelays \cite{appleby96,laub07}. Examples include Bacillus subtilis \cite{hoch00} and   yeast \cite{posas96}. A general network is depicted in Figure \ref{f.networks}(d). 
The {BIN} for the case $n=2$ can be written as:
\begin{align*} X_2+X_1^+ & \rightleftharpoons C_1 \rightleftharpoons  X_2^++X_1 \\
	 X_3+X_2^+ & \rightleftharpoons C_2 \rightleftharpoons  X_3^++X_2 \\
	 X_1 &\rightarrow X_1^+, ~ X_3^+ \rightarrow X_3.
\end{align*}
The pair $\Gamma, C$ can be written as:
{    \tiny \begin{align*} \Gamma&= \left[\begin{array}{rrrrrrrrrr} -1 & 0 & 0 & 1 & -1 & 0 & 0 & 0 & 0 & 0\\ 1 & -1 & 1 & 0 & 0 & 0 & 0 & 0 & 0 & 0\\ 0 & -1 & 1 & 0 & 0 & 0 & 0 & 1 & -1 & 0\\ 0 & 0 & 0 & 1 & -1 & -1 & 1 & 0 & 0 & 0\\ 0 & 0 & 0 & 0 & 0 & -1 & 1 & 0 & 0 & 1\\ 0 & 0 & 0 & 0 & 0 & 0 & 0 & 1 & -1 & -1\\ 0 & 1 & -1 & -1 & 1 & 0 & 0 & 0 & 0 & 0\\ 0 & 0 & 0 & 0 & 0 & 1 & -1 & -1 & 1 & 0 \end{array}\right], ~
		C  =\left[\begin{array}{rrrrrrrrrrr} -1 & 0 & 0 & 1 & -1 & 0 & 0 & 0 & 0 & 0\\ 1 & -1 & 1 & 0 & 0 & 0 & 0 & 0 & 0 & 0\\ 0 & -1 & 1 & 0 & 0 & 0 & 0 & 1 & -1 & 0\\ 0 & 0 & 0 & 1 & -1 & -1 & 1 & 0 & 0 & 0\\ 0 & 0 & 0 & 0 & 0 & -1 & 1 & 0 & 0 & 1\\ 0 & 0 & 0 & 0 & 0 & 0 & 0 & 1 & -1 & -1\\ 0 & 1 & -1 & -1 & 1 & 0 & 0 & 0 & 0 & 0\\ 0 & 0 & 0 & 0 & 0 & 1 & -1 & -1 & 1 & 0\\ -1 & 0 & 0 & 0 & 0 & 1 & -1 & 0 & 0 & 0\\ 1 & 0 & 0 & 0 & 0 & 0 & 0 & -1 & 1 & 0\\ 0 & 0 & 0 & -1 & 1 & 0 & 0 & 1 & -1 & 0\\ 0 & -1 & 1 & 0 & 0 & 1 & -1 & 0 & 0 & 0\\ 0 & -1 & 1 & 0 & 0 & 0 & 0 & 0 & 0 & 1\\ 0 & 0 & 0 & 1 & -1 & 0 & 0 & 0 & 0 & -1\\ -1 & 0 & 0 & 0 & 0 & 0 & 0 & 0 & 0 & 1 \end{array}\right]
		\end{align*} 
	} 
Let $\Lambda_1,\dots,\Lambda_{14}$ be the matrices that can {be} found using Corollary \ref{cor} {using the command \texttt{FindLambda($\Gamma$,C)}}. Let $\Xi=\sum_{\ell} \rho_\ell \Lambda_\ell$. Then, {using the output of the subroutine}, $\Xi$ can be written as:
 { \tiny  \[ \left[\begin{array}{rrrrrrrrrrrrccc} \Xi_{11} & -\rho_{12} & 0 & 0 & 0 & 0 & 0 & 0 & \rho_{6} & 0 & 0 & 0 & 0 & 0 & 0\\ -\rho_{11} & \Xi_{22} & 0 & 0 & 0 & 0 & -\rho_{1} & 0 & 0 & \rho_{4} & 0 & 0 & 0 & 0 & 0\\ 0 & 0 & \Xi_{33} & 0 & 0 & 0 & 0 & 0 & 0 & -\rho_{3} & \rho_{11} & \rho_{14} & \rho_{9} & 0 & 0\\ 0 & 0 & 0 & \Xi_{44} & 0 & 0 & 0 & 0 & -\rho_{2} & 0 & -\rho_{13} & -\rho_{12} & 0 & \rho_{8} & 0\\ 0 & 0 & 0 & 0 & \Xi_{55}& -\rho_{13} & 0 & -\rho_{10} & 0 & 0 & 0 & 0 & 0 & -\rho_{7} & 0\\ 0 & 0 & 0 & 0 & -\rho_{14} &\Xi_{66} & 0 & 0 & 0 & 0 & 0 & 0 & -\rho_{5} & 0 & 0\\ -\rho_{3} & -\rho_{2} & 0 & 0 & 0 & 0 &\Xi_{77} & 0 & 0 & 0 & \rho_{4} & -\rho_{6} & 0 & 0 & 0\\ 0 & 0 & 0 & 0 & -\rho_{9} & -\rho_{8} & 0 & \Xi_{88} & 0 & 0 & -\rho_{7} & \rho_{5} & 0 & 0 & 0\\ \rho_{7} & 0 & 0 & -\rho_{1} & 0 & 0 & 0 & 0 & \Xi_{99}& -\rho_{13} & 0 & 0 & 0 & 0 & \rho_{8}\\ 0 & \rho_{5} & 0 & 0 & 0 & 0 & 0 & 0 & -\rho_{14} & \Xi_{(10){10}}& -\rho_{1} & 0 & 0 & 0 & -\rho_{9}\\ 0 & 0 & \rho_{12} & -\rho_{14} & 0 & 0 & \rho_{5} & -\rho_{6} & 0 & -\rho_{2} &  \Xi_{(11){11}}& 0 & 0 & -\rho_{9} & 0\\ 0 & 0 & \rho_{13} & -\rho_{11} & 0 & 0 & -\rho_{7} & \rho_{4} & \rho_{3} & 0 & 0 &  \Xi_{(12){12}} & \rho_{8} & 0 & 0\\ 0 & 0 & \rho_{10} & 0 & 0 & -\rho_{4} & 0 & 0 & 0 & 0 & 0 & 0 & \Xi_{(13){13}} & -\rho_{11} & \rho_{3}\\ 0 & 0 & 0 & 0 & -\rho_{6} & 0 & 0 & 0 & 0 & 0 & -\rho_{10} & 0 & -\rho_{12} &  \Xi_{(14){14}}& -\rho_{2}\\ 0 & 0 & 0 & 0 & 0 & 0 & 0 & 0 & 0 & -\rho_{10} & 0 & 0 & 0 & -\rho_{1} &  \Xi_{(15){15}} \end{array}\right], \]}
where the diagonal elements are given as {\footnotesize $\diag(\Xi)=[-\rho_{1}-\rho_{2}-\rho_{6}-\rho_{12}, -\rho_{1}-\rho_{3}-\rho_{4}-\rho_{11}, -\rho_{3}-\rho_{4}-\rho_{5}-\rho_{9}-\rho_{11}-\rho_{14}, -\rho_{2}-\rho_{6}-\rho_{7}-\rho_{8}-\rho_{12}-\rho_{13}, -\rho_{7}-\rho_{8}-\rho_{10}-\rho_{13}, -\rho_{5}-\rho_{9}-\rho_{10}-\rho_{14}, -\rho_{2}-\rho_{3}-\rho_{4}-\rho_{6}-\rho_{11}-\rho_{12}, -\rho_{5}-\rho_{7}-\rho_{8}-\rho_{9}-\rho_{13}-\rho_{14}, -\rho_{1}-\rho_{7}-\rho_{8}-\rho_{13}, -\rho_{1}-\rho_{5}-\rho_{9}-\rho_{14}, -\rho_{2}-\rho_{5}-\rho_{6}-\rho_{9}-\rho_{12}-\rho_{14}, -\rho_{3}-\rho_{4}-\rho_{7}-\rho_{8}-\rho_{11}-\rho_{13}, -\rho_{3}-\rho_{4}-\rho_{10}-\rho_{11}, -\rho_{2}-\rho_{6}-\rho_{10}-\rho_{12}, -\rho_{1}-\rho_{10}] $}.

We have $\mu_\infty(\Xi)=0$, but the network can be verified {manually} to be weakly contractive with nonexpansivity depth 2. In particular, $S_-=\{1,\dots,8\}, S_{01}=\{9,\dots,14\}, S_{02}=\{15\}$. 
The weighting matrix $P_\theta B$ can be written as:
 {\tiny \[ P_\theta B = \frac 12 \left[\begin{array}{rrrrrrcc} 2 {\left(\theta +1\right)}^2 & 0 & 0 & 0 & 0 & 0 & 0 & 0\\ 0 & 2 {\left(\theta +1\right)}^2 & 0 & 0 & 0 & 0 & 0 & 0\\ 0 & 0 & 2 {\left(\theta +1\right)}^2 & 0 & 0 & 0 & 0 & 0\\ 0 & 0 & 0 & 2 {\left(\theta +1\right)}^2 & 0 & 0 & 0 & 0\\ 0 & 0 & 0 & 0 & 2 {\left(\theta +1\right)}^2 & 0 & 0 & 0\\ 0 & 0 & 0 & 0 & 0 & 2 {\left(\theta +1\right)}^2 & 0 & 0\\ 0 & 0 & 0 & 0 & 0 & 0 & 2 {\left(\theta +1\right)}^2 & 0\\ 0 & 0 & 0 & 0 & 0 & 0 & 0 & 2 {\left(\theta +1\right)}^2\\ 2 \theta +2 & 0 & 0 & -2 \theta -2 & 0 & 0 & 0 & 0\\ 0 & 2 \theta +2 & -2 \theta -2 & 0 & 0 & 0 & 0 & 0\\ 0 & 0 & \theta +1 & -\theta -1 & 0 & 0 & \theta +1 & -\theta -1\\ 0 & 0 & \theta +1 & -\theta -1 & 0 & 0 & -\theta -1 & \theta +1\\ 0 & 0 & 2 \theta +2 & 0 & 0 & -2 \theta -2 & 0 & 0\\ 0 & 0 & 0 & 2 \theta +2 & -2 \theta -2 & 0 & 0 & 0\\ 1 & -1 & 1 & -1 & 1 & -1 & 0 & 0 \end{array}\right].\] }
Hence, for an {arbitrary} given positive compact set in a stoichiometric class, the trajectories are strictly contractive with respect to the $P_\theta B$-weighted $\infty$-norm, where $\theta>0$ is small enough depending on the choice of the compact set. In addition, the trajectories  entrain to periodic time-varying parameters.

\section{Conclusions}
In this paper, we have shown that the graphical Lyapunov functions introduced in \cite{PWLRj,MA_LEARN}  automatically imply nonexpansiveness. Under additional conditions, they also imply strict contraction over arbitrary positive compact sets and  entrainment to periodic inputs. The provided tests can be checked computationally and are available on \texttt{github.com/malirdwi/LEARN}.

\appendix
\section*{Appendix: Additional Proofs}

\paragraph{Proof of Corollary \ref{cor}} 	The first statement is equivalent to the statement that $\kerGamma=\ker C$. Hence, we focus on the second statement. For the first direction, assume that the statement in Theorem \ref{th.metzler} holds. Fix $\ell$. Since $\tl\Lambda_\ell \in\mathbb R^{2m\times 2m}$, then we can write: (we drop the dependence on $\ell$ to simplify the notation)
\begin{equation}\label{e0}\tl\Lambda_\ell=\begin{bmatrix} \Lambda_{11} & \Lambda_{12} \\ \Lambda_{21} & \Lambda_{22} \end{bmatrix},\end{equation}
where $\Lambda_{11},\Lambda_{12},\Lambda_{21},\Lambda_{22}\in\mathbb R^{{m\times m}}$. By the statement in Theorem \ref{th.metzler}, $\Lambda_{11},\Lambda_{22}$ are Metzler and $\Lambda_{12},\Lambda_{21}$ are non-nonnegative. In addition, \eqref{e.metzler} can be written as:
\begin{align} \label{e2a} CQ_\ell = (\Lambda_{11}  - \Lambda_{12})  C \\ \label{e2b} CQ_\ell = (\Lambda_{22}  - \Lambda_{21} ) C \end{align}
Note that we can assume, w.l.o.g{.}, that $\Lambda_{22} =\Lambda_{11} $, $\Lambda_{21}=\Lambda_{12} $. Hence, we define $\Lambda_{\ell}:=\Lambda_{11}  - \Lambda_{12} $. It remains to show that $\mu_{\infty}(\Lambda_{\ell}) \le 0$. To that end, recall that $\mu_\infty(\Lambda_\ell)=\max_{i} \lambda_{ii}^{(\ell)}+\sum_{j\ne i}|\lambda_{ij}^{(\ell)}|$. Hence, we write:
\begin{align}\label{e1}\lambda_{ii}^{(\ell)}+\sum_j|\lambda_{ij}^{(\ell)}|= \lambda_{ii}^{11}-\lambda_{ii}^{12}+ \sum_{j\ne i} |\lambda_{ij}^{11}-\lambda_{ij}^{12}|\mathop{\le}^{(\star)} \lambda_{ii}^{11}+\lambda_{ii}^{12}+ \sum_{j\ne i} (\lambda_{ij}^{11}+\lambda_{ij}^{12}) \mathop{=}^{(\clubsuit)}0,
\end{align}
where the inequality $(\star)$ follows since $\lambda_{ii}^{12},\lambda_{ij}^{11},\lambda_{ij}^{12}$ are nonnegative, and equality $(\clubsuit)$ follows from  $\tl\Lambda_\ell \mathbf 1=0$. Hence, we have found a matrix $\Lambda_\ell$ that {satisfies} $\mu_\infty(\Lambda_\ell)\le 0$. 

Next, we  show the other direction. Fix $\ell$. Assume we have a matrix $\Lambda_\ell$ that satisfies $\mu_\infty(\Lambda_\ell)\le 0$.  Hence, for all $i$ we have $\lambda_{ii}^{(\ell)}\le0$ and ${\eta}_i:=-(\lambda_{ii}^{(\ell)}+\sum_j|\lambda_{ij}^{(\ell)}|)\ge0$.  We are going to construct two matrices $\Lambda_{11},\Lambda_{12}$ elementwise as follows assuming $i\ne j$: $\lambda_{ij}^{(11)}:=\max\{\lambda_{ij}^{(\ell)},0\}+{\eta}_i/{(2n)}, \lambda_{ij}^{(12)}:=\max\{-\lambda_{ij}^{(\ell)},0\}+{\eta}_i/{(2n)}$,  $\lambda_{ii}^{(12)}:={\eta}_i/{(2n)}$, $\lambda_{ii}^{(11)}:=\lambda_{ii}^{(\ell)}+{\eta}_i/{(2n)}$, where the latter is non-positive since $\lambda_{ii}^{(\ell)}+{\eta}_i/{(2n)}\le \lambda_{ii}^{(\ell)}+{\eta}_i \le \lambda_{ii}^{(\ell)}-\lambda_{ii}^{(\ell)}=0$.  Finally, we define $\Lambda_{22}:=\Lambda_{11},\Lambda_{21}:=\Lambda_{12}$. Define $\tl\Lambda_{\ell}$ as in \eqref{e0} which can be verified to be a Metzler matrix  satisfying  $\tl\Lambda_\ell \mathbf 1 = 0$ and \eqref{e2a}-\eqref{e2b}. \hfill $\blacksquare$

\paragraph{Proof of Theorem \ref{th.strict_contraction}}
Remember that $\rho_1,\dots,\rho_s$ correspond to the nonzero entries of $\partial R/\partial x$ which are strictly positive in $\mathbb R_+^n$ by AK3, hence the proof can be accomplished by showing that $\mu_{\infty}\left(\sum_{\ell=1}^s \rho_\ell    \Lambda_\ell \right )<0$ for any positive $\rho_1,\dots,\rho_s$. Assume $\Theta \in \mathbb R^{m\times m}$. Since $\Theta$ might have zero rows, the function can be written (by reordering species if necessary) as $\tV(r)=\|[\Theta_d,O] \Gamma r\|_{\infty}$, where $\Theta_d$ is an $n_d \times n_d$ positive diagonal matrix for some $n_d \le n$. Hence, $\tV(r)=\max_{i\in\{1,\dots,n_d\}}\gamma_i^T r$. To show the statement, it is sufficient to show that for each $i \in \{1,\dots,  n_d  \}$, there exists $\ell \in \{1,\dots,s\}$ such that $\lambda_\ell^{(ii)}  + \sum_{j \ne i} |\lambda_\ell^{(ik)}| <0$. \\
By AS1, $\gamma_i$ is orthogonal to a positive vector, and hence must have a strictly negative element, say $\gamma_{ij}$. Hence, this implies that $X_i$ is a reactant of $\mathbf R_j$, and $\partial R_j/\partial x_i$ is nonzero. Let $\Lambda_\ell$ {be} the matrix (defined in Corollary 8) corresponding to the reactant-{reaction} pair $(X_i,\R_j)$. Therefore, it can be seen it is possible to choose the $i$th row of $\Lambda_\ell$ as follows: $\lambda_\ell^{(ii)}= -\theta_{ii}$ and $\lambda_\ell^{(ij)}= 0, i \ne j$. Hence, $\lambda_\ell^{(ii)} + \sum_{j \ne i} |\lambda_\ell^{(ij)}| = -\theta_{ii} <0$. \strut \hfill $\blacksquare$

\section*{Acknowledgments}
M.A.A and E.D.S. {acknowledge} the support of grants AFOSR FA9550-21-1-0289 and FA9550-22-1-0316, and NSF/DMS-2052455.

\end{document}